\documentclass[11pt]{article}

\usepackage[T1]{fontenc}
\usepackage{lmodern}
\usepackage{graphicx} 

\usepackage{amsmath, amssymb, amsthm}
\usepackage[margin=2.5cm]{geometry}
\usepackage{color, graphicx, float, subfig}
\usepackage[breakable]{tcolorbox}
\usepackage{multirow}

\usepackage{accents}
\newcommand{\ubar}[1]{\underaccent{\bar}{#1}}

\makeatletter
\def\@tvsp{\mathchoice{{}\mkern-4.5mu}{{}\mkern-4.5mu}{{}\mkern-2.5mu}{}}
\def\ltrivert{\left|\@tvsp\left|\@tvsp\left|}
\def\rtrivert{\right|\@tvsp\right|\@tvsp\right|}
\makeatother
\newcommand\tnorm[1]{\ltrivert#1\rtrivert}

\usepackage{etoolbox}
\makeatletter
\patchcmd{\@maketitle}{\LARGE \@title}{\LARGE\bfseries\@title}{}{}

\renewcommand{\@seccntformat}[1]{\csname the#1\endcsname.\quad}
\makeatother

\definecolor{darkblue}{rgb}{0,0,.5}

\usepackage{hyperref}
\hypersetup{
	colorlinks=true,		
	linkcolor=darkblue,		
	citecolor=darkblue,		
	urlcolor=darkblue 		
}

\makeatletter
\def\th@plain{%
	\thm@notefont{}
	\itshape 
}
\def\th@definition{%
	\thm@notefont{}
	\normalfont 
}

\renewenvironment{proof}[1][\proofname]{\par
	\normalfont
	\topsep0\p@\@plus3\p@ \trivlist
	\item[\hskip\labelsep\itshape
	#1\@addpunct{.}]\ignorespaces
}{%
	\qed\endtrivlist
}
\makeatother

\newtheorem{theorem}{Theorem}[section]
\newtheorem{lemma}[theorem]{Lemma}

\newtheorem{proposition}[theorem]{Proposition}
\newtheorem{assumption}[theorem]{Assumption}
\theoremstyle{definition}

\theoremstyle{definition}
\newtheorem{example}[theorem]{Example}
\theoremstyle{definition}
\newtheorem{remark}[theorem]{Remark}
\theoremstyle{definition}
\newtheorem{algorithm}{Algorithm}

\renewcommand\theenumi{(\roman{enumi})}

\renewcommand{\labelenumi}{\rm (\roman{enumi})}

\usepackage[shortlabels]{enumitem}

\allowdisplaybreaks
\newcommand{\argmin}{\ensuremath{\operatorname*{argmin}}}
\newcommand{\dom}{\ensuremath{\operatorname{dom}}}
\newcommand{\gra}{\ensuremath{\operatorname{gra}}}

\newcommand{\prox}{\ensuremath{\operatorname{Prox}}}
\newcommand{\dist}{\ensuremath{\operatorname{dist}}}
\newcommand{\rank}{\ensuremath{\operatorname{rank}}}

\newcounter{step}[algorithm]
\newcommand\step[1]{%
	\refstepcounter{step}	
	\vskip 0.25\baselineskip
	\ifx\hfuzz#1\hfuzz
		\item[~\(\triangleright\)~\textbf{Step~\arabic{step}.}]
	\else
		\item[~\(\triangleright\)~\textbf{Step~\arabic{step}}] (\texttt{#1})\textbf{.}%
	\fi
}

\begin{document}

\title{Doubly relaxed forward-Douglas--Rachford splitting for the sum of two nonconvex and a DC function}

\author{
Minh N. Dao\thanks{School of Science, RMIT University, Melbourne, VIC 3000, Australia.
E-mail: \texttt{minh.dao@rmit.edu.au}.}, 
Tan Nhat Pham\thanks{Centre for New Energy Transition Research, Federation University Australia, Ballarat, VIC 3353, Australia.
E-mail: \texttt{pntan.iac@gmail.com}.}, 
and Phan Thanh Tung\thanks{Faculty of Mathematical Economics, Thuongmai University, Hanoi, Vietnam.
E-mail: \texttt{phanthanhtung@tmu.edu.vn}.}
}

\date{}

\maketitle

\begin{abstract}
In this paper, we consider a class of structured nonconvex nonsmooth optimization problems whose objective function is the sum of three nonconvex functions, one of which is expressed in a difference-of-convex (DC) form. This problem class covers several important structures in the literature including the sum of three functions and the general DC program. We propose a splitting algorithm and prove the subsequential convergence to a stationary point of the problem. The full sequential convergence, along with convergence rates for both the iterates and objective function values, is then established without requiring differentiability of the concave part. Our analysis not only extends but also unifies and improves recent convergence analyses in nonconvex settings. We benchmark our proposed algorithm with notable algorithms in the literature to show its competitiveness on a low rank matrix completion problem and a simutaneously sparse and low-rank matrix estimation problem. Our algorithm exhibits very competitive results compared to notable algorithms in the literature, on both synthetic data and public dataset.
\end{abstract}

{\small

\noindent{\bf Mathematics Subject Classification (MSC 2020):}
90C26,	
49M27,	
65K05.	
}

\section{Introduction}

We consider the composite optimization problem
\begin{align}\label{eq:prob}
\min_{x\in \mathcal{H}} F(x) :=f(x) +g(x) +h(x),
\end{align}
where $\mathcal{H}$ is a real Hilbert space, $f, g\colon \mathcal{H}\to \left(-\infty, +\infty\right]$ are proper lower semicontinuous functions, and $h =\bar{h} -\ubar{h}$ in which $\bar{h}\colon \mathcal{H}\to \mathbb{R}$ is a differentiable (possibly nonconvex) function and $\ubar{h}\colon \mathcal{H}\to \mathbb{R}$ is a continuous (possibly nonsmooth) convex function. This problem can be interpreted as minimizing a loss function $f$ over a constraint set (handled by $g$), with the incorporation of a regularization term $h$ that has a difference-of-convex (DC) form. The DC regularization has shown promising performance in the current literature, as shown in \cite{TLiu2019,Lou2017,Wen2017}.

If $\ubar{h} \equiv0$, then $h$ reduces to a differential function and problem \eqref{eq:prob} becomes a further extension of the commonly known sum of two functions in the literature. Many practical problems arising in machine learning or image processing can be formulated into minimizing a sum of a smooth function and a nonsmooth function, as shown in \cite{Cai2010,Chambolle2016}. The nonsmooth parts are usually called a regularization term, and they can be absolute-value norm ($\ell_1$-norm) or group lasso, which often has closed-form proximal operator or at least is fast to be computed. The success of those regularization terms led to the introduction of more complex, computational demanding regularization terms, such as the overlapping group lasso \cite{Jacob2009}, structured sparsity \cite{pmlr-v38-elhalabi15}, or total-variation \cite{barbero2018}. Fortunately, such complex terms can be decomposed into sum of terms which may have closed-form proximal operators. This leads to the problem of minimizing the sum of three functions. Several splitting algorithms have been developed in the literature to tackle this type of problem, both in convex and nonconvex settings. For the convex setting, a three-operator splitting scheme, which was later known as the \emph{Davis--Yin splitting}, was proposed in \cite{Davis2017} to solve the inclusion problem with three maximally monotone operators. This algorithm can be viewed as an extension of the well-known \emph{Douglas--Rachford} \cite{DR56} and \emph{forward-backward} schemes \cite{Passty_1979}. It can also be viewed as a generalization of the \emph{forward-Douglas--Rachford splitting}, which was studied in \cite{BriceoArias2013}. 
An adaptive three-operator splitting algorithm was developed in \cite{DP21} for the sum of two generalized monotone operators and one cocoercive operator. This also addresses the problem of minimizing the sum of three functions, two of which are weakly and strongly convex with the weak convexity neutralized. Existing studies for nonconvex settings are in fact limited in number. The general assumption is that at least two terms in \eqref{eq:prob} possess Lipschitz continuous gradients. Under such an assumption, the authors in \cite{Bian2021} established the subsequential and full sequential convergence of the Davis--Yin splitting with the help of the Kurdyka--{\L}ojasiewicz property.

When $f\equiv 0$ and other functions are assumed to be convex, problem \eqref{eq:prob} reduces to the general DC program
\begin{align}\label{eq:DC}
\min_{x\in \mathcal{H}} g(x) +\bar{h}(x) -\ubar{h}(x),
\end{align}
which is also a broad class of optimization problems in the literature \cite{Le_Thi2018-ht}. Notable algorithms to solve the DC program is the \emph{difference-of-convex algorithm} (DCA) and its variants as reported in \cite{Le_Thi2018-ht}. Recently, a splitting algorithm based on the Davis--Yin splitting has been proposed in \cite{Chuang2021} to solve the DC program with the assumption that the concave part has Lipschitz continuous gradient. This work suggests that the Davis--Yin splitting also has potential to solve the DC program; however, its convergence analysis relies on the strong convexity of $g$ and $\bar{h}$.

In this work, we propose a splitting algorithm tailored to the fully nonconvex problem \eqref{eq:prob}, which includes both minimizing the sum of three nonconvex functions and solving a generalized class of DC programs. Extending beyond the adaptive three-operator splitting in \cite{DP21}, our proposed algorithm not only features two flexible relaxation parameters but also possesses the ability to handle DC functions. We refer to it as the \emph{doubly relaxed forward-Douglas--Rachford splitting} (DRFDR). This algorithm exhibits the subsequential convergence to a stationary point of \eqref{eq:prob} under mild assumptions. With the Kurdyka--{\L}ojasiewicz property of a suitable function but without the differentiability of the concave part, we establish global convergence for the full sequence of iterates generated by the DRFDR, along with convergence rates for both the iterates and objective function values. We also thoroughly analyze the ranges of the important parameters used in the DRFDR and show that they are less restrictive than those used in previous works. Our convergence analysis unifies and enhances the existing convergence results for the Douglas--Rachford, Peaceman--Rachford, and Davis--Yin splitting algorithms in nonconvex settings \cite{LP16,LLP17,Bian2021}.

In the following examples, we present two problems that fit the structure of \eqref{eq:prob} along with their applications. We will also conduct numerical experiments on these problems to evaluate the performance of our proposed algorithm in a subsequent section of this paper.

\begin{example}[Low-rank matrix recovery]
\label{ex:MAT_REC}
Recovering an unknown low-rank or approximately low-rank matrix from very limited information is a critical problem arising in many fields, such as machine learning, control, or image processing \cite{Cai2010,JainSVP}. One simple example is recovering a data matrix from a sampling of its entries. This problem also has applications in power system operation \cite{Cai2018_IEEE,CSSE}, where it is used to recover the missing data from the phasor measurement units (PMUs), or load data from the supervisory control and data acquisition (SCADA) systems. The basic model of this problem is expressed as
\begin{align}\label{ex:mat_recover}
\min_{X\in \mathbb{R}^{m\times d}} \rank(X) \text{~~s.t.~~} \mathcal{P}_{\Omega} (X)= \mathcal{P}_{\Omega}(M), \tag{MR}
\end{align}
where $\Omega$ denotes the index set of matrix entries which are sampled uniformly, $\mathcal{P}_{\Omega}(\cdot)$ is the orthogonal projection onto the span of matrices vanishing outside the set $\Omega$, which means that the $(i,j)^{th}$ entry of $\mathcal{P}_{\Omega}(X)$ is equal to $X^{ij}$ if $(i,j) \in \Omega$ and $0$ otherwise. Since problem \eqref{ex:mat_recover} is NP-hard \cite{Meka2008}, a more robust formulation was proposed in \cite{JainSVP} as
\begin{align}\label{ex:mat_recover_re}
\min_{X\in \mathbb{R}^{m\times d}}~~\frac{1}{2}\|\mathcal{P}_{\Omega} (X)- \mathcal{P}_{\Omega}(M)\|^2 +\iota_{\mathcal{C}(r)}(X),  \tag{MR'}
\end{align}
where $\iota_{\mathcal{C}(r)}(\cdot)$ is the indicator function of $\mathcal{C}(r) :=\{X: \rank(X)\leq r\}$. To ensure the coerciveness of the objective function and the stability of the solutions, the authors in \cite{Bian2021} incorporated an additional Tikhonov regularization terms $\frac{\rho}{2}\|X\|_F^2$ ($\|\cdot\|_F$ is the standard Frobenius norm and $\rho \in (0, +\infty)$ is the regularization parameter) into \eqref{ex:mat_recover_re}. This led to the modified problem
\begin{align}\label{ex:mat_recover_mod}
\min_{X\in \mathbb{R}^{m\times d}}~~\frac{1}{2}\|\mathcal{P}_{\Omega} (X)- \mathcal{P}_{\Omega}(M)\|^2 +\iota_{\mathcal{C}(r)}(X) +\frac{\rho}{2}\|X\|_F^2.  \tag{MR-M}
\end{align}
Clearly, problem \eqref{ex:mat_recover_mod} fits the form of \eqref{eq:prob} with $f(X)=\frac{1}{2}\|\mathcal{P}_{\Omega} (X)- \mathcal{P}_{\Omega}(M)\|^2$, $g(X)=\iota_{\mathcal{C}(r)}(X)$, $\bar{h}(X)=\frac{\rho}{2}\|X\|_F^2$, and $\ubar{h} \equiv 0 $.
\end{example}

\begin{example}[Simultaneously sparse and low-rank matrix estimation]
\label{ex:low_rank_and_sparse} 
Matrix estimation has many modern applications such as covariance estimation, graphical model structure learning, graph denoising, or link prediction. The objective of this problem is to estimate an unknown matrix, which has a block-diagonal structure, from its noisy observation. As discussed in \cite{estimation_lrsm}, solutions of low-rank estimation problems are in general not sparse at all. Hence, it is important to introduce a combination of regularization terms to ensure that the solution is low-rank and sparse at the same time. A convex optimization model to solve this problem was introduced in \cite{estimation_lrsm} as
\begin{align}\label{prob:SLRME_1}
\min_{X\in \mathbb{R}^{m\times m}}~~\frac{1}{2}\|X-A\|_F^2 + \rho_1 \|X\|_1+ \rho_2 \|X\|_{*}, \tag{SLRME}
\end{align}
where $A \in \mathbb{R}^{m\times m}$ is the initial noisy input matrix, $\|X\|_{*}$ denotes the nuclear norm (which controls the rank of the solution), $\|X\|_1$ is the sum of the absolute values of all entries in $X$ (which controls the sparsity of the solution), and $\rho_1, \rho_2 \in (0, +\infty)$ are regularization parameters. Recently, a new representation for the nuclear norm has been proposed using the difference of squared Frobenius norm and squared Ky Fan $2$-$k$ norm \cite[Section~6.4.4]{Shi2022}. Based on this DC representation, we convert problem \eqref{prob:SLRME_1} into
\begin{align}\label{prob:SLRME_2}
\min_{X\in \mathbb{R}^{m\times m}}~~\frac{1}{2}\|X-A\|_F^2 + \rho_1 \|X\|_1+ \rho_2 (\|X\|_F^2 - \tnorm{X}^2_{k,2}) , \tag{SLRME-M}
\end{align}
in which $\tnorm{\cdot}_{k,2}$ denotes the Ky Fan $2$-$k$ norm. Now, problem \eqref{prob:SLRME_2} fits the form of \eqref{eq:prob} with $f(X)=\frac{1}{2}\|X-A\|_F^2$, $g(X)=\rho_1 \|X\|_1$, $\bar{h}(X)=\rho_2 \|X\|_F^2$, and $\ubar{h}(X)=\rho_2\tnorm{X}^2_{k,2}$. Therefore, it is worth investigating whether the nonconvex formulation \eqref{prob:SLRME_2} has any advantages over the convex counterpart \eqref{prob:SLRME_1}.
\end{example}

The rest of the paper is organized as follows. Section~\ref{sec:preliminaries} presents the preliminary materials used in this work. In Section~\ref{sec:algorithm}, we introduce our algorithm and then establish both subsequential and full sequential convergence under suitable assumptions. Section~\ref{sec:numerical_result} provides numerical results of the proposed algorithm. Finally, the conclusions are given in Section~\ref{sec:conclusion}.

\section{Preliminaries}
\label{sec:preliminaries}

We assume throughout that $\mathcal{H}$ is a real Hilbert space with inner product $\langle \cdot, \cdot\rangle$ and induced norm $\|\cdot\|$. The set of nonnegative integers is denoted by $\mathbb{N}$, the set of positive integers by $\mathbb{N}^\ast$, and the set of real numbers by $\mathbb{R}$. Given a set-valued operator $A\colon \mathcal{H}\rightrightarrows \mathcal{H}$, its \emph{domain} is $\dom A :=\{x\in \mathcal{H}: Ax\neq \varnothing\}$ and its \emph{graph} is $\gra A :=\{(x, u)\in \mathcal{H}\times \mathcal{H}: u\in Ax\}$.

Let $f\colon \mathcal{H}\to \left(-\infty, +\infty\right]$. The \emph{domain} of $f$ is $\dom f :=\{x\in \mathcal{H}: f(x) <+\infty\}$. We recall that $f$ is \emph{proper} if $\dom f\neq \varnothing$, \emph{lower semicontinuous} if, for all $x\in \mathcal{H}$, $f(x)\leq \liminf_{z\to x} f(z)$, and \emph{coercive} if $f(x)\to +\infty$ as $\|x\|\to +\infty$.
The function $f$ is \emph{$\alpha$-convex} for some $\alpha\in \mathbb{R}$ if, for all $x, y\in \dom f$ and all $\lambda\in (0, 1)$,
\begin{align*}
f((1 -\lambda)x +\lambda y) +\frac{\alpha}{2}\lambda(1 -\lambda)\|x -y\|^2\leq (1 -\lambda)f(x) +\lambda f(y).    
\end{align*}
We say that $f$ is \emph{convex} if $\alpha =0$, \emph{strongly convex} if $\alpha >0$, and \emph{weakly convex} if $\alpha <0$. It is known that $f$ is $\alpha$-convex if and only if $f -\frac{\alpha}{2}\|\cdot\|^2$ is convex.

Let $f\colon \mathcal{H}\to \left(-\infty, +\infty\right]$ be proper and lower semicontinuous. The \emph{regular subdifferential} of $f$ at $x\in \dom f$ is defined by
\begin{align*}
\widehat{\partial} f(x) :=\left\{x^*\in \mathcal{H}:\; \liminf_{y\to x}\frac{f(y) -f(x) -\langle x^*, y -x\rangle}{\|y -x\|}\geq 0\right\}    
\end{align*}
and the \emph{limiting subdifferential} of $f$ at $x\in \dom f$ is defined by
\begin{align*}
\partial f(x) :=\left\{x^*\in \mathcal{H}:\; \exists x_n\to x \text{~and~} x^*_n\in \widehat{\partial} f(x_n) \text{~with~} f(x_n)\to f(x) \text{~and~} x^*_n\to x^*\right\}.    
\end{align*}
Both regular and limiting subdifferentials of $f$ at $x\notin \dom f$ are defined to be the empty set. If $f$ is convex, then both regular and limiting subdifferentials coincide with the \emph{classical subdifferential} of convex analysis \cite[Theorem~1.93]{Mor06}, that is,
\begin{align}\label{eq:cvxsubdiff}
\widehat{\partial} f(x) =\partial f(x) =\{x^*\in \mathcal{H}:\; \forall y\in \mathcal{H},\; f(x) +\langle x^*, y -x\rangle \leq f(y)\}.
\end{align}
The function $f$ is said to satisfy the \emph{Kurdyka--{\L}ojasiewicz (KL) property} at $\overline{x} \in \dom \partial f$ if there exist $\varepsilon\in (0, +\infty)$, $\delta \in (0, +\infty]$, and a continuous concave function $\varphi: \left[0, \delta\right) \to [0, +\infty)$ such that $\varphi$ is continuously differentiable with $\varphi' > 0$ on $(0, \delta)$, $\varphi(0) = 0$, and
\begin{align*}
\varphi'(f(x) -f(\overline{x})) \dist(0, \partial f(x)) \geq 1,    
\end{align*}
whenever $\|x -\overline{x}\|\leq \varepsilon$ and $f(\overline{x}) < f(x) < f(\overline{x}) + \delta$. If $f$ satisfies the KL property at any $\overline{x} \in \dom \partial f$, then it is called a \emph{KL function}. We say that $f$ satisfies the \emph{KL property at $\overline{x}$ with exponent $\lambda \in [0, 1)$} if it satisfies the KL property at $\overline{x} \in \dom \partial f$ in which the corresponding function $\varphi$ can be chosen as $\varphi(t) = c t^{1 -\lambda}$ for some $c \in (0, +\infty)$. If $f$ is a KL function and has the same exponent $\lambda \in [0, 1)$ at any point in $\dom \partial f$, then it is called a \emph{KL function with exponent $\lambda$}.

We now collect some useful properties related to $\alpha$-convex functions.

\begin{proposition}
\label{p:subdiff}
Let $f\colon \mathcal{H}\to \left(-\infty, +\infty\right]$ be proper, lower semicontinuous, and $\alpha$-convex. Then the following hold:
\begin{enumerate}
\item\label{p:subdiff_mono}
$\partial f =\widehat{\partial}f$ is maximally $\alpha$-monotone in the sense that $(x, x^*)\in \gra \partial f$ if and only if, for all $(y, y^*)\in \gra \partial f$, $\langle y^* -x^*, y -x\rangle \geq \alpha\|y -x\|^2$. 
\item\label{p:subdiff_ineq}
For all $x, y\in \mathcal{H}$ and all $x^*\in \partial f(x)$, $f(x) +\langle x^*, y -x\rangle \leq f(y) -\frac{\alpha}{2}\|y -x\|^2$. Consequently, if $x$ is a local minimizer of $f$, then, for all $y\in \mathcal{H}$, $f(x) \leq f(y) -\frac{\alpha}{2}\|y -x\|^2$.
\end{enumerate}
\end{proposition}
\begin{proof}
By assumption, $h :=f -\frac{\alpha}{2}\|\cdot\|^2$ is a convex function. Using the subdifferential sum rule (\cite[Proposition~1.107]{Mor06}) and the convexity of $h$, we have that, for all $x\in \mathcal{H}$,
\begin{align}\label{eq:partial f}
\partial f(x) =\partial h(x) +\alpha x =\widehat{\partial}h(x) +\alpha x =\widehat{\partial}f(x).
\end{align}
So, $\partial f =\widehat{\partial}f$. The remaining conclusion of \ref{p:subdiff_mono} then follows from \cite[Lemma~5.2(i)]{DP19}.

Let $x, y\in \mathcal{H}$ and $x^*\in \partial f(x)$. In view of \eqref{eq:partial f}, $x^* =\tilde{x} +\alpha x$ for some $\tilde{x}\in \partial h(x)$. Since $h(x) +\langle \tilde{x}, y -x\rangle\leq h(y)$, we derive that 
\begin{align*}
f(x) +\langle x^*, y -x\rangle &=h(x) +\frac{\alpha}{2}\|x\|^2 +\langle \tilde{x} +\alpha x, y -x\rangle \\
&\leq h(y) +\frac{\alpha}{2}\|x\|^2 +\alpha\langle x, y -x\rangle \\
&=f(y) -\frac{\alpha}{2}\|y-x\|^2.
\end{align*} 
Finally, if $x$ is a local minimizer of $f$, then $0\in \partial f(x)$, and applying the above inequality with $x^* =0$ completes \ref{p:subdiff_ineq}. 
\end{proof}

\begin{proposition}
\label{p:diff}
Let $f\colon \mathcal{H}\to \left(-\infty, +\infty\right]$ be proper and let $\alpha\in \mathbb{R}$. Suppose that $\dom f$ is open and convex, and that $f$ is differentiable on $\dom f$. Then the following hold:
\begin{enumerate}
\item\label{p:diff_alpha-cvx} 
$f$ is $\alpha$-convex if and only if, for all $x, y\in \dom f$,
\begin{align*}
f(x) +\langle \nabla f(x), y -x\rangle \leq f(y) -\frac{\alpha}{2}\|y -x\|^2.
\end{align*} 
\item\label{p:diff_descent} 
If $\nabla f$ is $\kappa$-Lipschitz continuous on $\dom f$, then, for all $x, y\in \dom f$,
\begin{align*}
f(y) -\frac{\kappa}{2}\|y -x\|^2\leq f(x) +\langle \nabla f(x), y -x\rangle \leq f(y) +\frac{\kappa}{2}\|y -x\|^2
\end{align*}
and $f$ is $(-\kappa)$-convex.
\end{enumerate} 
\end{proposition}
\begin{proof}
\ref{p:diff_alpha-cvx}: Set $h :=f -\frac{\alpha}{2}\|\cdot\|^2$ and let $x, y\in \dom f$. Then $\dom h =\dom f$ and $\nabla f(x) =\nabla h(x) +\alpha x$. It follows that 
\begin{align*}
f(y) -f(x) -\langle \nabla f(x), y -x\rangle &=h(y) +\frac{\alpha}{2}\|y\|^2 -h(x) -\frac{\alpha}{2}\|x\|^2 -\langle \nabla h(x) +\alpha x, y -x\rangle \\
&=h(y) -h(x) -\langle \nabla h(x), y -x\rangle +\frac{\alpha}{2}\|y -x\|^2
\end{align*} 
The conclusion is now obtained by applying \cite[Proposition~17.7]{BC17} to $h$.

\ref{p:diff_descent}: This follows from \cite[Lemma~2.64(i)]{BC17} and \ref{p:diff_alpha-cvx}. Note that if $\nabla f$ is $0-$Lipschitz continuous on $\dom f$, then $f$ reduces to a linear function.
\end{proof}

We end this section the following technical lemma.
\begin{lemma}
\label{l:2products}
Let $a, b, u, v\in \mathcal{H}$. Then
\begin{align*}
\langle u -a, v -a\rangle =\langle u -b, v -b\rangle +\frac{1}{2}(\|u -a\|^2 -\|u -b\|^2) +\frac{1}{2}(\|v -a\|^2 -\|v -b\|^2).     
\end{align*}
\end{lemma}
\begin{proof}
We first have that, for $x\in \{a, b\}$, $\langle u -x, v -x\rangle =\frac{1}{2}(\|u -x\|^2 +\|v -x\|^2 -\|u -v\|^2)$, and the conclusion then follows.
\end{proof}

\section{Doubly relaxed forward-Douglas--Rachford splitting}
\label{sec:algorithm}

In this section, we propose a splitting algorithm with guaranteed global convergence to solve problem \eqref{eq:prob}. Recall that the \emph{proximity operator} of a proper function $f\colon \mathcal{H}\to \left(-\infty,+\infty\right]$ with parameter $\gamma\in (0, +\infty)$ at $x\in \mathcal{H}$ is defined by
\begin{align*}
\prox_{\gamma f}(x) :=\argmin_{z\in \mathcal{H}} \left(f(z) +\frac{1}{2\gamma}\|z -x\|^2\right).    
\end{align*}

\begin{tcolorbox}[
	left=0pt,right=0pt,top=0pt,bottom=0pt,
	colback=blue!10!white, colframe=blue!50!white,
  	boxrule=0.2pt,
  	breakable]
\begin{algorithm}[DRFDR]
\label{algo:DRFDR}
\step{}
Let $y_0 \in \mathcal{H}$, $z_0 \in \mathcal{H}$, and set $n =0$. Let $\gamma \in (0, +\infty)$, $\theta \in (0, 1]$, and $\eta \in (0, +\infty)$.

\step{}\label{step:main}
Compute $y^*_n \in \partial \ubar{h}(y_n)$ and find
\begin{align*}
\left\{
\begin{aligned}
&x_{n+1} \in \prox_{\gamma f}(z_n), \\
&y_{n+1} \in \prox_{\theta\gamma g}((\theta +1)x_{n+1} -\theta z_n -\theta\gamma \nabla \bar{h}(x_{n+1}) +\theta\gamma y^*_n), \\
&z_{n+1} =z_n +\eta(y_{n+1} -x_{n+1}).
\end{aligned}
\right.
\end{align*}

\step{}
If a termination criterion is not met, set $n =n+1$ and go to Step~\ref{step:main}.
\end{algorithm}
\end{tcolorbox}

\begin{remark}[Discussion of the algorithm structure]
\label{r:structure}
Let us take a closer look at Algorithm~\ref{algo:DRFDR} with some comments.
\begin{enumerate}
\item\label{r:structure_update}
The updating scheme of Algorithm~\ref{algo:DRFDR} can be written as computing $y^*_n \in \partial \ubar{h}(y_n)$ and finding
\begin{align*}
\left\{
\begin{aligned}
&x_{n+1} \in \argmin_{x\in \mathcal{H}} \left(f(x) +\frac{1}{2\gamma}\|x -z_n\|^2\right), \\
&y_{n+1} \in \argmin_{y\in \mathcal{H}} \left(g(y) +\frac{1}{2\theta\gamma}\|y -(\theta +1)x_{n+1} +\theta z_n +\theta\gamma \nabla \bar{h}(x_{n+1}) -\theta\gamma y^*_n\|^2\right), \\
&z_{n+1} =z_n +\eta(y_{n+1} -x_{n+1}).
\end{aligned}
\right.
\end{align*}
Since $y -(\theta +1)x_{n+1} +\theta z_n +\theta\gamma \nabla \bar{h}(x_{n+1}) -\theta\gamma y^*_n =(y -x_{n+1}) -\theta(x_{n+1} -z_n -\gamma \nabla \bar{h}(x_{n+1}) +\gamma y^*_n)$, we further derive that 
\begin{align*}
y_{n+1} &\in \argmin_{y\in \mathcal{H}} \bigg(g(y) -\frac{1}{\gamma}\langle x_{n+1} -z_n -\gamma \nabla \bar{h}(x_{n+1}) +\gamma y^*_n, y -x_{n+1}\rangle +\frac{1}{2\theta\gamma}\|y -x_{n+1}\|^2\bigg) \\
&= \argmin_{y\in \mathcal{H}} \bigg(g(y) +\bar{h}(x_{n+1}) +\langle \nabla \bar{h}(x_{n+1}), y -x_{n+1}\rangle -\ubar{h}(y_n) -\langle y^*_n, y -y_n \rangle \\
&\hspace{2cm} +\frac{1}{\gamma}\langle z_n -x_{n+1}, y -x_{n+1}\rangle +\frac{1}{2\theta\gamma}\|y -x_{n+1}\|^2\bigg). 
\end{align*}

\item 
When $\ubar{h} \equiv 0$, Step~\ref{step:main} of Algorithm~\ref{algo:DRFDR} becomes
\begin{align*}
\left\{
\begin{aligned}
&x_{n+1} \in \prox_{\gamma f}(z_n), \\
&y_{n+1} \in \prox_{\theta\gamma g}((\theta +1)x_{n+1} -\theta z_n -\theta\gamma \nabla \bar{h}(x_{n+1})), \\
&z_{n+1} =z_n +\eta(y_{n+1} -x_{n+1}).
\end{aligned}
\right.
\end{align*}
This is the adaptive splitting algorithm studied in \cite{DP21}, which reduces to the Davis--Yin splitting if $\theta =1$ and $\eta =1$.

\item 
When $h\equiv0$, Algorithm~\ref{algo:DRFDR} reduces to the \emph{adaptive Douglas--Rachford splitting} \cite{DP19}, which is written as
\begin{align*}
\left\{
\begin{aligned}
&x_{n+1} \in \prox_{\gamma f}(z_n), \\
&y_{n+1} \in \prox_{\theta\gamma g}((\theta +1)x_{n+1} -\theta z_n), \\
&z_{n+1} =z_n +\eta(y_{n+1} -x_{n+1}).
\end{aligned}
\right.
\end{align*}
This algorithm becomes the Douglas--Rachford splitting if $\theta =1$ and $\eta =1$, and the \emph{Peaceman--Rachford splitting} \cite{PR55} if $\theta =1$ and $\eta =2$.

\item
When $f\equiv 0$ and $\theta =1$, Step~\ref{step:main} of Algorithm~\ref{algo:DRFDR} reduces to
\begin{align*}
\left\{
\begin{aligned}
&y_{n+1} \in \prox_{\gamma g}(z_n -\gamma \nabla \bar{h}(z_n) +\gamma y^*_n) \text{~with~} y^*_n \in \partial \ubar{h}(y_n), \\
&z_{n+1} =(1 -\eta)z_n +\eta y_{n+1},
\end{aligned}
\right.
\end{align*}
which is a relaxed version of the \emph{generalized proximal point algorithm} proposed in \cite{An2016} for solving \eqref{eq:DC}. It is worthwhile noting that, when $\bar{h} \equiv 0$, Step~\ref{step:main} of Algorithm~\ref{algo:DRFDR} becomes
\begin{align*}
\left\{
\begin{aligned}
&x_{n+1} \in \prox_{\gamma f}(z_n), \\
&y_{n+1} \in \prox_{\theta\gamma g}((\theta +1)x_{n+1} -\theta z_n +\theta\gamma y^*_n) \text{~with~} y^*_n \in \partial \ubar{h}(y_n), \\
&z_{n+1} =z_n +\eta(y_{n+1} -x_{n+1}),
\end{aligned}
\right.
\end{align*}
providing an alternative algorithm for the general DC program \eqref{eq:DC} with $\bar{h}$ replaced by $f$. 

\item 
When $f\equiv0$, $\ubar{h} \equiv 0$, $\theta =1$, and $\eta =1$, the updating scheme of Algorithm~\ref{algo:DRFDR} simplifies to
\begin{align*}
\left\{
\begin{aligned}
&y_{n+1} \in \prox_{\gamma g}(z_n -\gamma \nabla \bar{h}(z_n)), \\
&z_{n+1} =(1 -\eta)z_n +\eta y_{n+1},
\end{aligned}
\right.
\end{align*}
which is the \emph{relaxed forward-backward algorithm}; see \cite[Section~4]{BDP22}.
\end{enumerate}
\end{remark}

From now on, assume that $\mathcal{H}$ is finite dimensional and let $(x_n, y_n, z_n, y^*_{n-1})_{n\in \mathbb{N}^\ast}$ be a sequence generated by Algorithm~\ref{algo:DRFDR}. We start our analysis by some useful properties of the sequence $(x_n, y_n, z_n, y^*_{n-1})_{n\in \mathbb{N}^\ast}$.

\begin{lemma}\label{l:OptCond}
The following hold: 
\begin{enumerate}
\item\label{l:OptCond_f} 
For all $n\in \mathbb{N}^\ast$, $0\in \partial f(x_n) -\frac{1}{\gamma}(z_{n-1} -x_n)$.
\item\label{l:OptCond_g} 
For all $n\in \mathbb{N}^\ast$, $0\in \partial g(y_n) +\nabla \bar{h}(x_n) -y^*_{n-1} +\frac{1}{\theta\gamma}(y_n -x_n) +\frac{1}{\gamma}(z_{n-1} -x_n)$.
\item\label{l:OptCond_cluster} 
If $(y_{n+1}, z_{n+1}) -(y_n, z_n)\to 0$ as $n\to +\infty$ and there exists a subsequence $(x_{k_n}, y_{k_n}, z_{k_n})_{n\in \mathbb{N}}$ converging to $(\overline{x}, \overline{y}, \overline{z})$, then $f(x_{k_n})\to f(\overline{x})$, $g(y_{k_n})\to g(\overline{y})$, $\overline{x} =\overline{y}$, and $0\in \partial f(\overline{y}) +\partial g(\overline{y}) +\nabla \bar{h}(\overline{y}) -\partial \ubar{h}(\overline{y})$. 
\end{enumerate}
\end{lemma}
\begin{proof}
\ref{l:OptCond_f} \& \ref{l:OptCond_g}: This follows from the optimality condition of the $x$- and $y$-updates.

\ref{l:OptCond_cluster}: By the $z$-update, for all $n\in \mathbb{N}$, $y_{n+1} -x_{n+1} =\frac{1}{\eta}(z_{n+1} -z_n)\to 0$. Since $(x_{k_n}, y_{k_n}, z_{k_n})\to (\overline{x}, \overline{y}, \overline{z})$ and $(y_{n+1}, z_{n+1}) -(y_n, z_n)\to 0$ as $n\to +\infty$, we derive that $\overline{x} =\overline{y}$, $y_{k_n-1}\to \overline{y}$, and $z_{k_n-1}\to \overline{z}$. As $\ubar{h}$ is a continuous convex function, it follows from \cite[Proposition~16.17]{BC17} that $(y^*_{k_n-1})_{n\in \mathbb{N}}$ is bounded. Passing to a subsequence if necessary, we can and do assume that $y^*_{k_n-1}\to \overline{y}^*\in \partial \ubar{h}(\overline{y})$ as $n\to +\infty$.   

Next, the $x$-update implies that 
\begin{align*}
f(x_{k_n}) +\frac{1}{2\gamma}\|x_{k_n} -z_{k_n-1}\|^2 \leq f(\overline{x}) +\frac{1}{2\gamma}\|\overline{x} -z_{k_n-1}\|^2
\end{align*} 
and the $y$-update implies that 
\begin{align*}
&g(y_{k_n}) +\frac{1}{2\theta\gamma}\|y_{k_n} -(\theta +1)x_{k_n} +\theta z_{k_n-1} +\theta\gamma \nabla \bar{h}(x_{k_n}) -\theta\gamma y^*_{k_n-1}\|^2 \\
&\leq g(\overline{y}) +\frac{1}{2\theta\gamma}\|\overline{y} -(\theta +1)x_{k_n} +\theta z_{k_n-1} +\theta\gamma \nabla \bar{h}(x_{k_n}) -\theta\gamma y^*_{k_n-1}\|^2.
\end{align*}
Letting $n\to +\infty$ and using the continuity of $\nabla h$ yield 
\begin{align*}
\limsup_{n\to +\infty} f(x_{k_n})\leq f(\overline{x}) \quad\text{and}\quad 
\limsup_{n\to +\infty} g(y_{k_n})\leq g(\overline{y}).
\end{align*}
Combining with the lower semicontinuity of $f$ and $g$ yields $\lim_{n\to +\infty} f(x_{k_n}) =f(\overline{x})$ and $\lim_{n\to +\infty} g(y_{k_n}) =g(\overline{y})$. Now, in view of \ref{l:OptCond_f} and \ref{l:OptCond_g}, $\frac{1}{\gamma}(z_{k_n-1} -x_{k_n})\in \partial f(x_{k_n})$ and $-\frac{1}{\theta\gamma}(y_{k_n} -x_{k_n}) -\frac{1}{\gamma}(z_{k_n-1} -x_{k_n}) \in \partial g(y_{k_n}) +\nabla \bar{h}(x_{k_n}) -y^*_{k_n-1}$.
Passing to the limit, we have $\frac{1}{\gamma}(\overline{z} -\overline{x})\in \partial f(\overline{x})$ and $-\frac{1}{\gamma}(\overline{z} -\overline{x})\in \partial g(\overline{y}) +\nabla \bar{h}(\overline{x}) -\overline{y}^* \subseteq \partial g(\overline{y}) +\nabla \bar{h}(\overline{x}) -\partial \ubar{h}(\overline{y})$, hence $0\in \partial f(\overline{x}) +\partial g(\overline{y}) +\nabla \bar{h}(\overline{x}) -\partial \ubar{h}(\overline{y})$. Since $\overline{x} =\overline{y}$, the proof is complete. 
\end{proof}

\begin{lemma}
\label{l:fdiff}
Suppose that $f$ is differentiable with $\kappa$-Lipschitz continuous gradient. Then, for all $n\in \mathbb{N}^\ast$, the following hold: 
\begin{enumerate}
\item\label{l:fdiff_z} 
$z_{n-1} =x_n +\gamma\nabla f(x_n)$. 
\item\label{l:fdiff_zz}
$\|z_n -z_{n-1}\| \leq (1 +\gamma\kappa)\|x_{n+1} -x_n\|$.
\item\label{l:fdiff_yy}
$\|y_{n+1} -y_n\| \leq \frac{1 +\gamma\kappa}{\eta}\|x_{n+2} -x_{n+1}\| +\left(1 +\frac{1 +\gamma\kappa}{\eta}\right)\|x_{n+1} -x_n\|$.
\end{enumerate} 
\end{lemma}
\begin{proof}
\ref{l:fdiff_z}: This follows from Lemma~\ref{l:OptCond}\ref{l:OptCond_f} and the differentiability of $f$.

\ref{l:fdiff_zz}: We have from \ref{l:fdiff_z} that $z_n -z_{n-1} =(x_{n+1} -x_n) +\gamma(\nabla f(x_{n+1}) -\nabla f(x_n))$ and the conclusion then follows from the Lipschitz continuity of $\nabla f$.

\ref{l:fdiff_yy}: By the $z$-update, $y_{n+1} -y_n =\frac{1}{\eta}(z_{n+1} -z_n) -\frac{1}{\eta}(z_n -z_{n-1}) +(x_{n+1} -x_n)$, which together with \ref{l:fdiff_zz} completes the proof.
\end{proof}

In the upcoming analyses, we will rely on the following assumption.
\begin{assumption}
[Standing assumptions]
\label{a:standing}
~
\begin{enumerate}
\item\label{a:standing_f}
$f\colon \mathcal{H}\to \mathbb{R}$ a differentiable $\alpha$-convex function with $\kappa$-Lipschitz continuous gradient and $\alpha \geq -\kappa$.
\item 
$g\colon \mathcal{H}\to \left(-\infty, +\infty\right]$ is a proper lower semicontinuous function. 
\item 
$h =\bar{h} -\ubar{h}$, where $\bar{h}\colon \mathcal{H}\to \mathbb{R}$ is a differentiable function with $\ell$-Lipschitz continuous gradient, and $\ubar{h}\colon \mathcal{H}\to \mathbb{R}$ is a continuous convex function.
\end{enumerate}
\end{assumption}

\begin{remark}[Relationship between $\alpha$ and $\kappa$]
\label{r:alphakappa}
Regarding Assumption~\ref{a:standing}\ref{a:standing_f}, we note from Proposition~\ref{p:diff}\ref{p:diff_descent} that a differentiable function with $\kappa$-Lipschitz continuous gradient is always $(-\kappa)$-convex. Therefore, the assumption on $\alpha$-convexity of $f$ only makes sense if $\alpha \geq -\kappa$.

By Proposition~\ref{p:subdiff}\ref{p:subdiff_mono}, it follows from Assumption~\ref{a:standing}\ref{a:standing_f} that, for all $x, y\in \mathcal{H}$, $\alpha\|y -x\|^2 \leq \langle \nabla f(y) -\nabla f(x), y -x\rangle \leq \kappa\|y -x\|^2$. As a result, $\alpha \leq \kappa$. 
\end{remark} 

Let us now consider
\begin{align}\label{eq:Lyapunov}
L_n &= f(x_n) +g(y_n) +\bar{h}(x_n) +\langle \nabla \bar{h}(x_n), y_n -x_n\rangle -\ubar{h}(y_{n-1}) \notag \\
&\quad   -\langle y^*_{n-1}, y_n -y_{n-1}\rangle +\frac{1}{\gamma}\langle z_n -x_n, y_n -x_n\rangle -\frac{2\eta\theta -1}{2\theta\gamma}\|y_n -x_n\|^2
\end{align}
and
\begin{align*}
\varphi(\gamma) =2\theta\kappa(\kappa +\ell)\gamma^2 -\left((\eta\theta +2 -2\theta)\alpha -(3\eta -2)\theta\ell\right)\gamma +\eta -2.
\end{align*}
The following lemma gives a sufficient decrease as well as lower and upper bounds for the sequence $(L_n)_{n\in \mathbb{N}^\ast}$. In particular, its proof also explains how $(L_n)_{n\in \mathbb{N}^\ast}$ is constructed.

\begin{lemma}[Sufficient decrease]
\label{l:Lfun}
Suppose that Assumption~\ref{a:standing} holds and that $(\eta -1)\ell \geq 0$. Then, for all $n\in \mathbb{N}^\ast$, the following hold:
\begin{enumerate}
\item\label{l:Lfun_decrease}
$L_{n+1} -L_n \leq \frac{\varphi(\gamma)}{2\eta\theta\gamma}\|x_{n+1} -x_n\|^2$. Consequently, the sequence $(L_n)_{n\in \mathbb{N}^\ast}$ is nonincreasing provided that $\varphi(\gamma) <0$.
\item\label{l:Lfun_lb} 
$L_n \geq F(y_n) +\frac{1}{2}\left(\frac{1}{\theta\gamma} -\kappa -\ell\right)\|y_n -x_n\|^2$.
\item\label{l:Lfun_ub}
$L_n \leq F(y_n) +\langle y^*_n -y^*_{n-1}, y_n -y_{n-1}\rangle +\frac{1}{2}\left(\frac{1}{\theta\gamma} +\kappa +\ell\right)\|y_n -x_n\|^2$.   
\end{enumerate}
\end{lemma}
\begin{proof}
Let $n\in \mathbb{N}$. As $x_{n+1}$ is a minimizer of the $(\alpha +\frac{1}{\gamma})$-convex function $f +\frac{1}{2\gamma}\|\cdot -z_n\|^2$, it follows from Proposition~\ref{p:subdiff}\ref{p:subdiff_ineq} that
\begin{align}\label{eq:fx+}
&f(x_{n+1}) +\frac{1}{2\gamma}\|x_{n+1} -z_n\|^2 \notag \\
&\leq f(x_n) +\frac{1}{2\gamma}\|x_n -z_n\|^2 -\left(\frac{\alpha}{2} +\frac{1}{2\gamma}\right)\|x_{n+1} -x_n\|^2.
\end{align}
In view of Remark~\ref{r:structure}\ref{r:structure_update},
\begin{align}\label{eq:gy+}
&g(y_{n+1}) +\bar{h}(x_{n+1}) +\langle \nabla \bar{h}(x_{n+1}), y_{n+1} -x_{n+1}\rangle -\ubar{h}(y_n) -\langle y^*_n, y_{n+1} -y_n \rangle \notag \\
&\qquad +\frac{1}{\gamma}\langle z_n -x_{n+1}, y_{n+1} -x_{n+1}\rangle +\frac{1}{2\theta\gamma}\|x_{n+1} -y_{n+1}\|^2 \notag \\
&\leq g(y_n) +\bar{h}(x_{n+1}) +\langle \nabla \bar{h}(x_{n+1}), y_n -x_{n+1}\rangle -\ubar{h}(y_n) \notag \\ 
&\qquad +\frac{1}{\gamma}\langle z_n -x_{n+1}, y_n -x_{n+1}\rangle +\frac{1}{2\theta\gamma}\|y_n -x_{n+1}\|^2.
\end{align}
We note from the $z$-update that
\begin{align}\label{eq:x+-y+}
&\langle z_n -x_{n+1}, y_{n+1} -x_{n+1}\rangle \notag \\
&= \langle z_{n+1} -x_{n+1}, y_{n+1} -x_{n+1}\rangle -\langle z_{n+1} -z_n, y_{n+1} -x_{n+1}\rangle \notag \\
&= \langle z_{n+1} -x_{n+1}, y_{n+1} -x_{n+1}\rangle -\eta\|y_{n+1} -x_{n+1}\|^2,
\end{align}
and from Lemma~\ref{l:2products} that
\begin{align}\label{eq:x+-y}
\langle z_n -x_{n+1}, y_n -x_{n+1}\rangle &= \langle z_n -x_n, y_n -x_n\rangle +\frac{1}{2}(\|x_{n+1} -z_n\|^2 -\|x_n -z_n\|^2) \notag \\
&\hspace{2cm} +\frac{1}{2}(\|y_n -x_{n+1}\|^2 -\|y_n -x_n\|^2).
\end{align}
For $\omega \in (0, +\infty)$,
\begin{align}\label{eq:nabla h}
&\bar{h}(x_{n+1}) +\langle \nabla \bar{h}(x_{n+1}), y_n -x_{n+1}\rangle \notag \\
&= \bar{h}(x_{n+1}) +\langle \nabla \bar{h}(x_n), y_n -x_{n+1}\rangle +\langle \nabla \bar{h}(x_{n+1}) -\nabla \bar{h}(x_n), y_n -x_{n+1}\rangle \notag \\
&= \bar{h}(x_{n+1}) +\langle \nabla \bar{h}(x_n), x_n -x_{n+1}\rangle +\langle \nabla \bar{h}(x_n), y_n -x_n\rangle \notag \\
&\quad +\langle \nabla \bar{h}(x_{n+1}) -\nabla \bar{h}(x_n), y_n -x_{n+1}\rangle \notag \\
&\leq \bar{h}(x_n) +\frac{\ell}{2}\|x_{n+1} -x_n\|^2 +\langle \nabla \bar{h}(x_n), y_n -x_n\rangle +\ell\|x_{n+1} -x_n\|\|y_n -x_{n+1}\| \notag \\
&\leq \bar{h}(x_n) +\langle \nabla \bar{h}(x_n), y_n -x_n\rangle +\left(\frac{\ell}{2} +\frac{\ell\omega}{2}\right)\|x_{n+1} -x_n\|^2 +\frac{\ell}{2\omega}\|y_n -x_{n+1}\|^2,
\end{align}
where the first inequality holds due to the Lipschiz continuity of $\nabla h$ and Proposition~\ref{p:diff}\ref{p:diff_descent}, while the last one follows from Young inequality. 

Now, let $n\in \mathbb{N}^\ast$. Since $y^*_{n-1} \in \partial \ubar{h}(y_{n-1})$, the convexity of $\ubar{h}$ along with \eqref{eq:cvxsubdiff} yields
\begin{align}\label{eq:p cvx}
\ubar{h}(y_{n-1}) +\langle y^*_{n-1}, y_n -y_{n-1}\rangle \leq \ubar{h}(y_n).
\end{align}
We deduce from \eqref{eq:fx+}, \eqref{eq:gy+}, \eqref{eq:x+-y+}, \eqref{eq:x+-y}, \eqref{eq:nabla h}, and \eqref{eq:p cvx} that
\begin{align}\label{eq:LL}
L_{n+1} -L_n &\leq \left(-\frac{1}{2\gamma} -\frac{\alpha}{2} +\frac{\ell}{2} +\frac{\ell\omega}{2}\right)\|x_{n+1}-x_n\|^2 \notag \\
&\quad +\left(\frac{\ell}{2\omega} +\frac{\theta +1}{2\theta\gamma}\right)\|y_n -x_{n+1}\|^2 +\frac{2\eta\theta -\theta -1}{2\theta\gamma}\|y_n -x_n\|^2.
\end{align}
Next, it follows from the $z$-update and Lemma~\ref{l:fdiff}\ref{l:fdiff_z} that
\begin{align*}
y_n -x_n = \frac{1}{\eta}(z_n -z_{n-1}) = \frac{1}{\eta}(x_{n+1} -x_n) +\frac{\gamma}{\eta}(\nabla f(x_{n+1}) -\nabla f(x_n)),
\end{align*}
and so
\begin{align*}
y_n -x_{n+1} &= (y_n -x_n) -(x_{n+1} -x_n) \\
&= -\frac{\eta -1}{\eta}(x_{n+1} -x_n) +\frac{\gamma}{\eta}(\nabla f(x_{n+1}) -\nabla f(x_n)).
\end{align*}
Since $f$ is $\alpha$-convex, we have from Proposition~\ref{p:subdiff}\ref{p:subdiff_mono} that
\begin{align*}
\langle x_{n+1} -x_n, \nabla f(x_{n+1}) -\nabla f(x_n)\rangle \geq \alpha\|x_{n+1} -x_n\|^2.    
\end{align*}  
By combining with \eqref{eq:LL} and the Lipschitz continuity of $\nabla f$, we derive that
\begin{align*}
&L_{n+1} -L_n\\
&\leq \left(-\frac{1}{2\gamma} -\frac{\alpha-\ell-\ell\omega}{2}  +\frac{(\eta -1)^2\ell}{2\eta^2\omega} +\frac{(\eta -1)^2(\theta +1)}{2\eta^2\theta\gamma} +\frac{2\eta\theta -\theta -1}{2\eta^2\theta\gamma}\right)\|x_{n+1}-x_n\|^2 \\
&\quad -\left(\frac{(\eta -1)\ell\gamma}{\eta^2\omega} +\frac{(\eta -1)(\theta +1)}{\eta^2\theta} -\frac{2\eta\theta -\theta -1}{\eta^2\theta}\right)\langle x_{n+1} -x_n, \nabla f(x_{n+1}) -\nabla f(x_n)\rangle \\
&\quad +\left(\frac{\ell}{2\omega} +\frac{\theta +1}{2\theta\gamma} +\frac{2\eta\theta -\theta -1}{2\theta\gamma}\right)\frac{\gamma^2}{\eta^2}\|\nabla f(x_{n+1}) -\nabla f(x_n)\|^2 \\
&= \left(\frac{\eta -2}{2\eta\theta\gamma} -\frac{\alpha}{2} +\frac{\ell}{2} +\frac{\ell\omega}{2} +\frac{(\eta -1)^2\ell}{2\eta^2\omega}\right)\|x_{n+1}-x_n\|^2 \\
&\quad -\left(\frac{(\eta -1)\ell\gamma}{\eta^2\omega} +\frac{1 -\theta}{\eta\theta}\right)\langle x_{n+1} -x_n, \nabla f(x_{n+1}) -\nabla f(x_n)\rangle \\
&\quad +\left(\frac{\ell\gamma^2}{2\eta^2\omega} +\frac{\gamma}{\eta}\right)\|\nabla f(x_{n+1}) -\nabla f(x_n)\|^2 \\
&\leq \Psi(\gamma)\|x_{n+1}-x_n\|^2,  
\end{align*}
where
\begin{align*}
\Psi(\gamma) &:= \frac{\eta -2}{2\eta\theta\gamma} -\frac{\alpha}{2} +\frac{\ell}{2} +\frac{\ell\omega}{2} +\frac{(\eta -1)^2\ell}{2\eta^2\omega} -\frac{(\eta -1)\alpha\ell\gamma}{\eta^2\omega} -\frac{(1 -\theta)\alpha}{\eta\theta} +\frac{\kappa^2\ell\gamma^2}{2\eta^2\omega} +\frac{\kappa^2\gamma}{\eta} \\
&= \frac{\eta -2}{2\eta\theta\gamma} -\frac{(\eta\theta +2 -2\theta)\alpha}{2\eta\theta} +\frac{\ell}{2} +\ell\psi(\omega) +\frac{\gamma\kappa^2}{\eta}
\end{align*}
with
\begin{align*}
\psi(\omega) :=\frac{\omega}{2} +\frac{(\eta -1)^2}{2\eta^2\omega} -\frac{(\eta -1)\alpha\gamma}{\eta^2\omega} +\frac{\kappa^2\gamma^2}{2\eta^2\omega} =\frac{\omega}{2} +\frac{(\eta -1 +\gamma\kappa)^2 -2(\eta -1)(\alpha +\kappa)\gamma}{2\eta^2\omega}.
\end{align*}
If $\ell >0$, then $\eta \geq 1$ and, since $\alpha \geq -\kappa$, we can choose $\omega =\frac{\eta -1 +\gamma\kappa +\sqrt{2(\eta -1)(\alpha +\kappa)\gamma}}{\eta} >0$ to obtain $\psi(\omega) =\frac{\eta -1 +\kappa\gamma}{\eta}$, which implies that
\begin{align*}
\Psi(\gamma) &=\frac{\eta -2}{2\eta\theta\gamma} -\frac{(\eta\theta +2 -2\theta)\alpha}{2\eta\theta} +\frac{\ell}{2} +\frac{(\eta -1 +\kappa\gamma)\ell}{\eta} +\frac{\gamma\kappa^2}{\eta} \\
&= \frac{\eta -2}{2\eta\theta\gamma} -\frac{(\eta\theta +2 -2\theta)\alpha}{2\eta\theta} +\frac{(3\eta -2)\ell}{2\eta} +\frac{\kappa(\kappa +\ell)\gamma}{\eta} \\
&= \frac{\eta -2 -\left((\eta\theta +2 -2\theta)\alpha -(3\eta -2)\theta\ell\right)\gamma +2\theta\kappa(\kappa +\ell)\gamma^2}{2\eta\theta\gamma}.
\end{align*}
As the latter also holds when $\ell =0$, we get \ref{l:Lfun_decrease}.

To get \ref{l:Lfun_lb} and \ref{l:Lfun_ub}, we first have from the $z$-update and Lemma~\ref{l:fdiff}\ref{l:fdiff_z} that
\begin{align*}
z_n =z_{n-1} +\eta(y_n -x_n) =x_n +\gamma \nabla f(x_n) +\eta(y_n -x_n).   
\end{align*}
Substituting this into \eqref{eq:Lyapunov} yields
\begin{align}\label{eq:Lyapunov'}
L_n &=f(x_n) +\langle \nabla f(x_n), y_n -x_n\rangle +g(y_n) +\bar{h}(x_n) +\langle \nabla \bar{h}(x_n), y_n -x_n\rangle \notag \\ 
&\quad -\ubar{h}(y_{n-1}) -\langle y^*_{n-1}, y_n -y_{n-1}\rangle +\frac{1}{2\theta\gamma}\|y_n -x_n\|^2.
\end{align}
In view of Proposition~\ref{p:diff}\ref{p:diff_descent}, the Lipschitz continuity of $\nabla f$ and $\nabla h$ gives
\begin{align*}
f(y_n) -\frac{\kappa}{2}\|y_n -x_n\|^2 &\leq f(x_n) +\langle \nabla f(x_n), y_n -x_n\rangle \leq f(y_n) +\frac{\kappa}{2}\|y_n -x_n\|^2, \\
\bar{h}(y_n) -\frac{\ell}{2}\|y_n -x_n\|^2 &\leq \bar{h}(x_n) +\langle \nabla \bar{h}(x_n), y_n -x_n\rangle \leq \bar{h}(y_n) +\frac{\ell}{2}\|y_n -x_n\|^2.
\end{align*}
As $y^*_{n-1}\in \partial \ubar{h}(y_{n-1})$ and $y^*_n\in \partial \ubar{h}(y_n)$, it follows from the convexity of $\ubar{h}$ that
\begin{align*}
\ubar{h}(y_n) -\langle y^*_n, y_n -y_{n-1}\rangle \leq \ubar{h}(y_{n-1}) \leq \ubar{h}(y_n) -\langle y^*_{n-1}, y_n -y_{n-1}\rangle
\end{align*}
Combining these inequalities with \eqref{eq:Lyapunov'}, we get \ref{l:Lfun_lb} and \ref{l:Lfun_ub}.
\end{proof}

As observed in Lemma~\ref{l:Lfun}, to obtain the sufficient decrease property of $(L_n)_{n\in \mathbb{N}^\ast}$, we require that $\varphi(\gamma) <0$. Let us now explore the condition for $\gamma$ further.

\begin{remark}[Condition for parameter $\gamma$]
\label{r:gamma}
We first examine the scenario where $\kappa =0$. Then $\alpha =0$ (see Remark~\ref{r:alphakappa}) and
\begin{align*}
\varphi(\gamma) <0 \iff (3\eta -2)\theta\ell\gamma +\eta -2 <0 \iff \left(\eta \leq \frac{2}{3}\right) \text{~~or~~} \left(\frac{2}{3} < \eta < 2 \text{~and~} \ell\gamma < \frac{2 -\eta}{(3\eta -2)\theta}\right).    
\end{align*}

Next, assume that $\kappa >0$. Then $\varphi$ is a quadratic polynomial of $\gamma$ with leading coefficient $2\theta\kappa(\kappa +\ell) >0$ and discriminant $\Delta :=\left((\eta\theta +2 -2\theta)\alpha -(3\eta -2)\theta\ell\right)^2 -8(\eta -2)\theta\kappa(\kappa +\ell)$. Therefore,
\begin{align*}
\varphi(\gamma) <0 \iff (\Delta >0 \text{~and~} \underline{\gamma} <\gamma <\overline{\gamma}),     
\end{align*}
where $\underline{\gamma} :=\frac{(\eta\theta +2 -2\theta)\alpha -(3\eta -2)\theta\ell -\sqrt{\Delta}}{4\theta\kappa(\kappa +\ell)}$ and $\overline{\gamma} :=\frac{(\eta\theta +2 -2\theta)\alpha -(3\eta -2)\theta\ell +\sqrt{\Delta}}{4\theta\kappa(\kappa +\ell)}$. 
\begin{enumerate}
\item
Given that $\gamma\in (0, +\infty)$, we consider the following cases.

\emph{Case 1:} $\eta <2$. Then $\Delta >0$, $\overline{\gamma} >0$, and
\begin{align*}
(\varphi(\gamma) <0 \text{~and~} \gamma\in (0, +\infty)) \iff \gamma\in (0, \overline{\gamma}).     
\end{align*}

\emph{Case 2:} $\eta \geq 2$. Then $\underline{\gamma}\overline{\gamma} =\frac{\eta -2}{2\theta\kappa(\kappa +\ell)} \geq 0$ provided that $\Delta >0$. Thus,
\begin{align*}
(\Delta >0 \text{~and~} \overline{\gamma} >0) &\iff \left(\Delta >0 \text{~and~}\frac{(\eta\theta +2 -2\theta)\alpha -(3\eta -2)\theta\ell}{2\theta\kappa(\kappa +\ell)} =\underline{\gamma} +\overline{\gamma} >0\right) \\
&\iff (\eta\theta +2 -2\theta)\alpha -(3\eta -2)\theta\ell >\sqrt{8(\eta -2)\theta\kappa(\kappa +\ell)}.
\end{align*}
Since $\eta\theta +2 -2\theta >0$, we deduce that 
\begin{align*}
(\varphi(\gamma) <0 \text{~and~} \gamma\in (0, +\infty)) \iff \left(\alpha >\frac{(3\eta -2)\theta\ell +2\sqrt{2(\eta -2)\theta\kappa(\kappa +\ell)}}{\eta\theta +2 -2\theta} \text{~and~} \gamma\in (\underline{\gamma}, \overline{\gamma})\right),     
\end{align*}
where the condition for $\alpha$ is ensured when either $\alpha >0$ and $\theta$ is sufficiently small, or $\alpha >2\theta\ell$ and $\eta$ is sufficiently close to $2$. We should also keep in mind that $\alpha \leq \kappa$.

\item
As seen later in Theorem~\ref{t:subcvg}, to guarantee the boundedness of the sequence generated by Algorithm~\ref{algo:DRFDR}, it is essential to require $\gamma <\frac{1}{\theta(\kappa +\ell)}$. We observe that, since $\alpha \leq \kappa$,
\begin{align*}
\varphi\left(\frac{1}{\theta(\kappa +\ell)}\right) &= \frac{2\kappa}{\theta(\kappa +\ell)} -\frac{(\eta\theta +2 -2\theta)\alpha -(3\eta -2)\theta\ell}{\theta(\kappa +\ell)} +\eta -2 \\
&\geq \frac{2\kappa -(\eta\theta +2 -2\theta)\kappa +(3\eta -2)\theta\ell +(\eta -2)\theta(\kappa +\ell)}{\theta(\kappa +\ell)} \\
&= \frac{4(\eta -1)\ell}{\kappa +\ell}.
\end{align*}
Suppose that $(\eta -1)\ell \geq 0$. Then $\varphi\left(\frac{1}{\theta(\kappa +\ell)}\right)\geq 0$. Hence, in the case when $\Delta >0$, we must have $\frac{1}{\theta(\kappa +\ell)} \notin (\underline{\gamma}, \overline{\gamma})$, and so
\begin{align*}
\underline{\gamma} < \frac{1}{\theta(\kappa +\ell)} &\iff \overline{\gamma} \leq \frac{1}{\theta(\kappa +\ell)} \\ 
&\iff \frac{\eta -2}{2\theta\kappa(\kappa +\ell)} =\underline{\gamma}\overline{\gamma} < \left(\frac{1}{\theta(\kappa +\ell)}\right)^2 \\
&\iff \eta <2 +\frac{2\kappa}{\theta(\kappa +\ell)}.
\end{align*}
\end{enumerate}
\end{remark}

We now arrive at the main convergence results, indicating that the DRFDR indeed offers flexibility with two adjustable relaxation parameters $\theta$ and $\eta$. Even when these parameters are set to match those of existing algorithms, our proposed algorithm allows for a larger range of the stepsize parameter $\gamma$.

\begin{theorem}[Subsequential convergence]
\label{t:subcvg}
Suppose that Assumption~\ref{a:standing} holds, that $F$ is coercive, and that either
\begin{enumerate}
\renewcommand\theenumi{(\alph{enumi})}
\renewcommand{\labelenumi}{\rm (\alph{enumi})}
\item
$\kappa >0$, $\eta \in (0, 2)$ if $\ell =0$, $\eta \in [1, 2)$ if $\ell >0$, and $\gamma \in (0, \overline{\gamma})$; or
\item 
$\kappa >0$, $\eta \in \left[2, 2 +\frac{2\kappa}{\theta(\kappa +\ell)}\right)$, $\alpha >\frac{(3\eta -2)\theta\ell +2\sqrt{2(\eta -2)\theta\kappa(\kappa +\ell)}}{\eta\theta +2 -2\theta}$, and $\gamma\in (\underline{\gamma}, \overline{\gamma})$; or
\item
$\kappa =0$, $\eta \in (0, 2)$, $\gamma \in \left(0, \frac{1}{\theta\ell}\right)$ if $\eta \leq 1$, and\footnote{Here, we adopt the convention that $\frac{1}{\theta\ell} =\frac{2 -\eta}{(3\eta -2)\theta\ell} =+\infty$ when $\ell =0$.} $\gamma \in \left(0, \frac{2 -\eta}{(3\eta -2)\theta\ell}\right)$ if $\eta > 1$,
\end{enumerate}
where 
\begin{align*}
\underline{\gamma} :=\frac{(\eta\theta +2 -2\theta)\alpha -(3\eta -2)\theta\ell -\sqrt{\Delta}}{4\theta\kappa(\kappa +\ell)}
\text{~and~}
\overline{\gamma} :=\frac{(\eta\theta +2 -2\theta)\alpha -(3\eta -2)\theta\ell +\sqrt{\Delta}}{4\theta\kappa(\kappa +\ell)}    
\end{align*}
with $\Delta :=\left((\eta\theta +2 -2\theta)\alpha -(3\eta -2)\theta\ell\right)^2 -8(\eta -2)\theta\kappa(\kappa +\ell)$. Then the following hold:
\begin{enumerate}
\item\label{t:subcvg_bounded} 
The sequence $(x_n, y_n, z_n, y^*_{n-1})_{n\in \mathbb{N}^\ast}$ is bounded and
\begin{align*}
\sum_{n=1}^{+\infty} \|(x_{n+1}, y_{n+1}, z_{n+1}) -(x_n, y_n, z_n)\|^2 <+\infty.
\end{align*}
Consequently, $(x_{n+1}, y_{n+1}, z_{n+1}) -(x_n, y_n, z_n)\to 0$ as $n\to +\infty$.

\item\label{t:subcvg_cluster}
For every cluster point $(\overline{x}, \overline{y}, \overline{z})$ of $(x_n, y_n, z_n)_{n\in \mathbb{N}^\ast}$, it holds that $\overline{x} =\overline{y}$,
\begin{align*}
\lim_{n\to +\infty} L_n =\lim_{n\to +\infty} F(y_n) =F(\overline{y}), \text{~~and~~} 0\in \partial (f +g +\bar{h})(\overline{y}) -\partial \ubar{h}(\overline{y}).
\end{align*} 
\end{enumerate}
\end{theorem}
\begin{proof}
We first have from Remark~\ref{r:gamma} that $\varphi(\gamma) <0$ and $(\kappa +\ell)\gamma <\frac{1}{\theta}$. By Lemma~\ref{l:Lfun}\ref{l:Lfun_decrease}\&\ref{l:Lfun_lb},
\begin{align}\label{eq:Ln}
\text{the sequence $(L_n)_{n\in \mathbb{N}^\ast}$ is nonincreasing}
\end{align}
and, for all $n\in \mathbb{N}^\ast$,
\begin{align}\label{eq:bounded}
L_1\geq L_n \geq F(y_n) +\frac{1}{2}\left(\frac{1}{\theta\gamma} -\kappa -\ell\right)\|y_n -x_n\|^2.
\end{align}
Since $F$ is proper, lower semicontinuous, and coercive, it is bounded below due to \cite[Theorem~1.9]{RW98}. Moreover, $\frac{1}{\theta\gamma} -\kappa -\ell >0$ because $(\kappa +\ell)\gamma <\frac{1}{\theta}$. In view of \eqref{eq:bounded}, $(F(y_n))_{n\in \mathbb{N}^\ast}$ and $(\|y_n -x_n\|)_{n\in \mathbb{N}^\ast}$ are bounded. Combining with the coercivity of $F$,  this implies the boundedness of $(y_n)_{n\in \mathbb{N}^\ast}$ and then of $(x_n)_{n\in \mathbb{N}^\ast}$. Next, the boundedness of $(z_n)_{n\in \mathbb{N}^\ast}$ follows from the continuity of $\nabla f$ and the fact that $z_n =x_{n+1} +\gamma \nabla f(x_{n+1})$. Since $y^*_{n-1}\in \partial \ubar{h}(y_{n-1})$ and $\ubar{h}$ is continuous and convex, we use \cite[Proposition~16.20]{BC17} to get the boundedness of $(y^*_{n-1})_{n\in \mathbb{N}^\ast}$.

Since both terms on the right-hand side of \eqref{eq:bounded} are bounded below, so is $(L_n)_{n\in \mathbb{N}^\ast}$. Together with \eqref{eq:Ln}, we obtain that $(L_n)_{n\in \mathbb{N}^\ast}$ converges to some $\overline{L} \in \mathbb{R}$. Now, by again using Lemma~\ref{l:Lfun}\ref{l:Lfun_decrease} and telescoping sum,
\begin{align*}
-\frac{\varphi(\gamma)}{2\eta\theta\gamma}\sum_{n=1}^{+\infty} \|x_{n+1} -x_n\|^2 \leq \sum_{n=1}^{+\infty} (L_n -L_{n+1}) =L_1 -\overline{L} <+\infty, 
\end{align*}    
and thus, $\sum_{n=1}^{+\infty} \|x_{n+1} -x_n\|^2 <+\infty$. This combined with Lemma~\ref{l:fdiff}\ref{l:fdiff_zz}\&\ref{l:fdiff_yy} yields
\begin{align*}
\sum_{n=1}^{+\infty} \|(x_{n+1}, y_{n+1}, z_{n+1}) -(x_n, y_n, z_n)\|^2 <+\infty,
\end{align*}
and so $(x_{n+1}, y_{n+1}, z_{n+1}) -(x_n, y_n, z_n)\to 0$ as $n\to +\infty$. 

Let $(\overline{x}, \overline{y}, \overline{z})$ be a cluster point of $(x_n, y_n, z_n)_{n\in \mathbb{N}^\ast}$. Then there exists a subsequence $(x_{k_n}, y_{k_n}, z_{k_n})_{n\in \mathbb{N}}$ converging to $(\overline{x}, \overline{y}, \overline{z})$. We derive from Lemma~\ref{l:OptCond}\ref{l:OptCond_cluster} that $\overline{x} =\overline{y}$, 
\begin{align*}
0\in \nabla f(\overline{y}) +\partial g(\overline{y}) +\nabla \bar{h}(\overline{y}) -\partial \ubar{h}(\overline{y}) =\partial(f +g +\bar{h})(\overline{y}) -\partial \ubar{h}(\overline{y}),    
\end{align*}
and $g(y_{k_n})\to g(\overline{y})$ as $n\to +\infty$. Note that
\begin{align*}
L_{k_n} &= f(x_{k_n}) +g(y_{k_n}) +\bar{h}(x_{k_n}) +\langle \nabla \bar{h}(x_{k_n}), y_{k_n} -x_{k_n}\rangle -\ubar{h}(y_{k_n-1})  \notag \\
&\quad -\langle y^*_{k_n-1}, y_{k_n} -y_{k_n-1}\rangle +\frac{1}{\gamma}\langle z_{k_n} -x_{k_n}, y_{k_n} -x_{k_n}\rangle -\frac{2\eta\theta -1}{2\theta\gamma}\|y_{k_n} -x_{k_n}\|^2.    
\end{align*}
Using the boundedness of $(y^*_{n-1})_{n\in \mathbb{N}^\ast}$ and the continuity of $f$, $\bar{h}$, $\nabla \bar{h}$, and $\ubar{h}$, we have $L_{k_n}\to f(\overline{y}) +g(\overline{y}) +\bar{h}(\overline{y}) -\ubar{h}(\overline{y}) =F(\overline{y})$ as $n\to +\infty$. By the convergence of the whole sequence $(L_n)_{n\in \mathbb{N}^\ast}$, it follows that
\begin{align*}
\lim_{n\to +\infty} L_n =F(\overline{y}).    
\end{align*}
Finally, since $(y^*_{n-1})_{n\in \mathbb{N}^\ast}$ is bounded and $y_n -x_n =\frac{1}{\eta}(z_n -z_{n-1})\to 0$, $y_n -y_{n-1}\to 0$ as $n\to +\infty$, we deduce from Lemma~\ref{l:Lfun}\ref{l:Lfun_lb}\&\ref{l:Lfun_ub} that $\lim_{n\to +\infty} F(y_n) =\lim_{n\to +\infty} L_n$, completing the proof.
\end{proof}

\begin{remark}[Larger range of parameter $\gamma$]
Regarding the condition for $\gamma$ in Theorem~\ref{t:subcvg}, we consider the following cases.
\begin{enumerate}
\item
$\eta =1$. Then the condition for $\gamma$ becomes 
\begin{align*}
\gamma <\overline{\gamma} =\frac{(2 -\theta)\alpha -\theta\ell +\sqrt{\left((2 -\theta)\alpha -\theta\ell\right)^2 +8\theta\kappa(\kappa +\ell)}}{4\theta\kappa(\kappa +\ell)}, \end{align*}
which is equivalent to $\varphi(\gamma) =2\theta\kappa(\kappa +\ell)\gamma^2 -\left((2 -\theta)\alpha -\theta\ell\right)\gamma -1 <0$ (see Remark~\ref{r:gamma}).  
If additionally $\theta =1$, then the condition reduces to
\begin{align*}
\gamma <\frac{\alpha -\ell +\sqrt{(\alpha -\ell)^2 +8\kappa(\kappa +\ell)}}{4\kappa(\kappa +\ell)},     
\end{align*}
or equivalently,
\begin{align}\label{eq:gamma for DY}
\varphi(\gamma) =2\kappa(\kappa +\ell)\gamma^2 -\left(\alpha -\ell\right)\gamma -1 <0.
\end{align}
This considerably improves the convergence result for the Davis--Yin splitting (Algorithm~\ref{algo:DRFDR} with $\ubar{h} \equiv 0$, $\theta =1$, and $\eta =1$) analyzed in \cite{Bian2021}. Indeed, the latter requires
\begin{align*}
\frac{1}{2}\left(\frac{1}{\gamma} +\alpha\right) -\ell -\left(\frac{1}{\gamma} +\frac{\ell}{2}\right)\left((-1 -2\alpha\gamma) +(1 +\kappa\gamma)^2\right) >0,
\end{align*}
which can be written as
\begin{align}\label{eq:gamma for DY'}
\varphi(\gamma) + (\kappa\gamma -1)^2\ell\gamma +2(\kappa -\alpha)\ell\gamma^2 +4(\kappa -\alpha)\gamma <0.
\end{align}
Since $\alpha \leq \kappa$, \eqref{eq:gamma for DY'} implies \eqref{eq:gamma for DY} and moreover, by Remark~\ref{r:gamma}, $\gamma <\frac{1}{\kappa +\ell} \leq \frac{1}{\kappa}$, which yields $(\kappa\gamma -1)^2 >0$. Therefore, \eqref{eq:gamma for DY'} even strictly more restrictive than \eqref{eq:gamma for DY} once $\ell >0$ or $\alpha <\kappa$. 
 
Let us now consider the case when $\bar{h} \equiv 0$ (which yields $\ell =0$), $\theta =1$, and $\eta =1$. Then the condition for $\gamma$ becomes
\begin{align}\label{eq:gamma for DR}
\gamma <\frac{\alpha +\sqrt{\alpha^2 +8\kappa^2}}{4\kappa^2}, \text{~or equivalently,~} 2\kappa^2\gamma^2 -\alpha\gamma -1 <0.
\end{align}
If additionally $\ubar{h} \equiv 0$, then Algorithm~\ref{algo:DRFDR} reduces to the Douglas--Rachford splitting whose convergence is established in \cite{LP16} under the setting of Theorem~\ref{t:subcvg} except that 
\begin{align*}
(1 +\kappa\gamma)^2 -\frac{5\alpha\gamma}{2} -\frac{3}{2} <0, \text{~or equivalently,~} (2\kappa^2\gamma^2 -\alpha\gamma -1) +4(\kappa -\alpha)\gamma <0,
\end{align*}
which is more restrictive than \eqref{eq:gamma for DR} since $\alpha \leq \kappa$.

\item 
$\eta =2$. Then we require that $\alpha >2\theta\ell$ and
\begin{align*}
\gamma <\frac{\alpha -2\theta\ell}{\theta\kappa(\kappa +\ell)}.    
\end{align*}
If, in addition, $\bar{h} \equiv 0$ (which yields $\ell =0$), these requirements reduce to $\alpha >0$ ($f$ is strongly convex) and $\gamma <\frac{\alpha}{\theta\kappa^2}$. This significantly improves the result for the Peaceman--Rachford splitting (Algorithm~\ref{algo:DRFDR} with $\bar{h} \equiv 0$, $\ubar{h} \equiv 0$, $\theta =1$, and $\eta =2$) in \cite{LLP17} which requires $\alpha >\frac{3}{2}\kappa$ and $\gamma <\frac{3\alpha -2\kappa}{\kappa^2}$. 
\end{enumerate}
\end{remark}

Now, we recall that the \emph{conjugate} of a proper function $p$ is $p^*\colon \mathcal{H}\to (-\infty, +\infty]$ given by
\begin{align*}
p^*(u) =\sup_{x\in \mathcal{H}} (\langle u, x\rangle -p(x)).
\end{align*}
According to, e.g., \cite[Proposition~16.10]{BC17}, for all $y\in \mathcal{H}$ and all $y^*\in \partial \ubar{h}(y)$,
\begin{align}\label{eq:conjugate}
-\ubar{h}(y) =\ubar{h}^*(y^*) -\langle y^*, y\rangle \text{~~and~~} y\in \partial \ubar{h}^*(y^*).
\end{align}
The former together with \eqref{eq:Lyapunov} implies that, for all $n\in \mathbb{N}^\ast$,
\begin{align*}
L_n &= f(x_n) +g(y_n) +\bar{h}(x_n) +\langle \nabla \bar{h}(x_n), y_n -x_n\rangle +\ubar{h}^*(y^*_{n-1}) -\langle y^*_{n-1}, y_n\rangle \notag \\
&\hspace{4.5cm} +\frac{1}{\gamma}\langle z_n -x_n, y_n -x_n\rangle -\frac{2\eta\theta -1}{2\theta\gamma}\|y_n -x_n\|^2 \\
&= \mathcal{L}(x_n, y_n, z_n, \nabla \bar{h}(x_n), y^*_{n-1}),
\end{align*}
where $\mathcal{L}\colon \mathcal{H}^5\to \left(-\infty, +\infty\right]$ defined by
\begin{align*}
\mathcal{L}(x, y, z, u, v) &=f(x) +g(y) +\bar{h}(x) +\langle u, y -x\rangle +\ubar{h}^*(v) -\langle v, y\rangle \notag \\
&\hspace{2cm} +\frac{1}{\gamma}\langle z -x, y -x\rangle -\frac{2\eta\theta -1}{2\theta\gamma}\|y -x\|^2.
\end{align*}

\begin{theorem}[Full sequential convergence]
\label{t:fullcvg}
Suppose that the assumptions of Theorem~\ref{t:subcvg} hold and that $\mathcal{L}$ is a KL function. Then
\begin{align*}
\sum_{n=1}^{+\infty} \|(x_{n+1}, y_{n+1}, z_{n+1}) -(x_n, y_n, z_n)\| <+\infty.
\end{align*}
and the sequence $(x_n, y_n, z_n)_{n\in \mathbb{N}^\ast}$ converges to $(\tilde{x}, \tilde{y}, \tilde{z})$ with $\tilde{x} =\tilde{y}$ and
\begin{align*}
0\in \partial (f +g +\bar{h})(\tilde{y}) -\partial \ubar{h}(\tilde{y}).
\end{align*}
Moreover, if $\mathcal{L}$ is a KL function with exponent $\lambda\in [0, 1)$, then exactly one of the following alternatives holds: 
\begin{enumerate}
\item\label{t:fullcvg_finite}
\emph{(Finite convergence)} $\lambda =0$ and there exists $n_0\in \mathbb{N}$ such that, for all $n\geq n_0$, $(x_n, y_n, z_n) =(\tilde{x}, \tilde{y}, \tilde{z})$.
\item\label{t:fullcvg_linear}
\emph{(Linear convergence)} $\lambda \leq \frac{1}{2}$ and there exist $\Gamma \in (0, +\infty)$ and $\rho \in (0, 1)$ such that, for all $n\in \mathbb{N}^\ast$, $\|(x_n, y_n, z_n) -(\tilde{x}, \tilde{y}, \tilde{z})\|\leq \Gamma\rho^\frac{n}{2}$ and $|F(y_n) -F(\tilde{y})|\leq \Gamma\rho^\frac{n}{2}$.
\item\label{t:fullcvg_sublinear}
\emph{(Sublinear convergence)} $\lambda >\frac{1}{2}$ and there exists $\Gamma \in (0, +\infty)$ such that, for all $n\in \mathbb{N}^\ast$, $\|(x_n, y_n, z_n) -(\tilde{x}, \tilde{y}, \tilde{z})\|\leq \Gamma n^{-\frac{1-\lambda}{2\lambda-1}}$ and $|F(y_n) -F(\tilde{y})|\leq \Gamma n^{-\frac{1-\lambda}{2\lambda-1}}$.
\end{enumerate}
\end{theorem}
\begin{proof}
For each $n\in \mathbb{N}^\ast$, set $w_n =(x_n, y_n, z_n, \nabla \bar{h}(x_n), y^*_{n-1})$. Let $n\in \mathbb{N}^\ast$. We derive from Lemma~\ref{l:Lfun}\ref{l:Lfun_decrease} that
\begin{align*}
\mathcal{L}(w_{n+1}) -\frac{\varphi(\gamma)}{2\eta\theta\gamma}\|x_{n+1} -x_n\|^2 \leq \mathcal{L}(w_n),
\end{align*}
and from Theorem~\ref{t:subcvg}\ref{t:subcvg_bounded} that $(x_n, y_n, z_n, y^*_{n-1})_{n\in \mathbb{N}^\ast}$ is bounded, and so is $(w_n)_{n\in \mathbb{N}^\ast}$ due to the continuity of $\nabla h$. For every cluster point $\overline{w}$ of $(w_n)_{n\in \mathbb{N}^\ast}$, one has $\overline{w} =(\overline{x}, \overline{y}, \overline{z}, \nabla \bar{h}(\overline{x}), \overline{y}^*)$ with $\overline{y}^*\in \partial \ubar{h}(\overline{y})$ and, by Theorem~\ref{t:subcvg}\ref{t:subcvg_cluster} along with the equality in \eqref{eq:conjugate}, $\overline{x} =\overline{y}$ and
\begin{align*}
\mathcal{L}(w_n) =L_n \to F(\overline{y}) =\mathcal{L}(\overline{w}) \text{~~as~~} n\to +\infty.    
\end{align*}

Next, we see that 
\begin{align*}
\partial \mathcal{L}(w_n) 
=\begin{bmatrix}
\nabla f(x_n) +\nabla \bar{h}(x_n) -\nabla \bar{h}(x_n) -\frac{1}{\gamma}(z_n -x_n) -\frac{1}{\gamma}(y_n -x_n) +\frac{2\eta\theta -1}{\theta\gamma}(y_n -x_n) \\
\partial g(y_n) +\nabla \bar{h}(x_n) -y^*_{n-1} +\frac{1}{\gamma}(z_n -x_n) -\frac{2\eta\theta -1}{\theta\gamma}(y_n -x_n) \\
\frac{1}{\gamma}(y_n -x_n) \\
y_n -x_n \\
\partial \ubar{h}^*(y^*_{n-1}) -y_n
\end{bmatrix}.
\end{align*}
On the other hand, we have from Lemma~\ref{l:fdiff}\ref{l:fdiff_z} that $\nabla f(x_n) =\frac{1}{\gamma}(z_{n-1} -x_n)$, from Lemma~\ref{l:OptCond}\ref{l:OptCond_g} that $-\frac{1}{\theta\gamma}(y_n -x_n) -\frac{1}{\gamma}(z_{n-1} -x_n)\in \partial g(y_n) +\nabla \bar{h}(x_n) -y^*_{n-1}$, from the $z$-update that $y_n -x_n =\frac{1}{\eta}(z_n -z_{n-1})$, and from \eqref{eq:conjugate} that $y_{n-1}\in \partial \ubar{h}^*(y^*_{n-1})$. This leads to
\begin{align*}
\begin{bmatrix}
\frac{\eta\theta -\theta -1}{\eta\theta\gamma}(z_n -z_{n-1}) \\
-\frac{1}{\gamma}(z_n -z_{n-1}) \\
\frac{1}{\eta\gamma}(z_n -z_{n-1}) \\
\frac{1}{\eta}(z_n -z_{n-1}) \\
-(y_n -y_{n-1})
\end{bmatrix}
\in \partial \mathcal{L}(w_n),
\end{align*}
By combining with Lemma~\ref{l:fdiff}\ref{l:fdiff_zz}\&\ref{l:fdiff_yy}, there exists $C\in (0, +\infty)$ such that, for all $n\geq 2$,
\begin{align*}
\dist(0, \partial \mathcal{L}(w_n))\leq C(\|x_{n+1} -x_n\| +\|x_n -x_{n-1}\|).
\end{align*}
We now apply \cite[Theorem~5.1]{BDL22} with $\Delta_n=\|x_{n+1}-x_n\|$, $\alpha_n \equiv -\frac{\varphi(\gamma)}{2\eta\theta\gamma}$, $I =\{0,1\}$, $\lambda_1 =\lambda_2 =1/2$, $\beta_n \equiv 1/(2C)$, and $\varepsilon_n \equiv 0$ to obtain that
\begin{align*}
\sum_{n=1}^{+\infty} \|x_{n+1}-x_{n}\| =\sum_{n=1}^{+\infty} \Delta_n < +\infty,
\end{align*}
which together with Lemma~\ref{l:fdiff}\ref{l:fdiff_zz}\&\ref{l:fdiff_yy} implies that
\begin{align*}
\sum_{n=0}^{+\infty}\|(x_{n+1},y_{n+1},z_{n+1})-(x_{n},y_{n},z_{n})\| < +\infty.
\end{align*}
Therefore, $(x_{n},y_{n},z_{n},)_{n \in \mathbb{N}}$ converges to a point $(\tilde{x},\tilde{y},\tilde{z})\in \mathcal{H}^3$. By Theorem~\ref{t:subcvg}\ref{t:subcvg_cluster}, $\tilde{x} =\tilde{y}$ and $0\in \partial (f +g +\bar{h})(\tilde{y}) -\partial \ubar{h}(\tilde{y})$.

Now, assume that $\mathcal{L}$ is a KL function with exponent $\lambda\in [0, 1)$. We distinguish the following three cases.

\emph{Case 1:} $\lambda =0$. Following the arguments in the proof of \cite[Theorem~5.1]{BDL22} and \cite[Theorem~2(i)]{Attouch2007}, we get \ref{t:fullcvg_finite}.

\emph{Case 2:} $\lambda \leq \frac{1}{2}$. In view of \cite[Theorem~5.1(iv)]{BDL22}, there exist $\Gamma_0 \in (0, +\infty)$ and $\rho \in (0, 1)$ such that, for all $n\in \mathbb{N}^\ast$,
\begin{align*}
\|x_n -\tilde{x}\| \leq \Gamma_0\rho^\frac{n}{2} \text{~~and~~} |\mathcal{L}(w_n) -\mathcal{L}(\tilde{w})|\leq \Gamma_0\rho^n.    
\end{align*}
Let $n\in \mathbb{N}^\ast$. Then
\begin{align*}
\|x_{n+1} -x_n\| \leq \|x_{n+1} -\tilde{x}\| +\|x_n -\tilde{x}\|\leq \Gamma_0\rho^\frac{n+1}{2} +\Gamma_0\rho^\frac{n}{2} <2\Gamma_0\rho^\frac{n}{2},
\end{align*}
which combined with Lemma~\ref{l:fdiff}\ref{l:fdiff_zz}\&\ref{l:fdiff_yy} implies that
\begin{align}\label{eq:zzyy}
\|z_n -z_{n-1}\| \leq 2(1 +\gamma\kappa)\Gamma_0\rho^\frac{n}{2}
\text{~~and~~}
\|y_{n+1} -y_n\| \leq 2\left(1 +\frac{2(1 +\gamma\kappa)}{\eta}\right)\Gamma_0\rho^\frac{n}{2}.
\end{align}
By the $z$-update,
\begin{align}\label{eq:ynxn}
\|y_n -x_n\| =\frac{1}{\eta}\|z_n -z_{n-1}\| \leq \frac{2(1 +\gamma\kappa)}{\eta}\Gamma_0\rho^\frac{n}{2}.    
\end{align}
As $\tilde{x} =\tilde{y}$, it follows that
\begin{align*}
\|y_n -\tilde{y}\| \leq \|x_n -\tilde{x}\| +\|y_n -x_n\| \leq \left(1 +\frac{2(1 +\gamma\kappa)}{\eta}\right)\Gamma_0\rho^\frac{n}{2}.    
\end{align*}
Recall from Lemma~\ref{l:fdiff}\ref{l:fdiff_z} that $z_{n-1} =x_n +\gamma \nabla f(x_n)$. Passing to the limit yields $\tilde{z} =\tilde{x} +\gamma \nabla f(\tilde{x})$, from which and the Lipschitz continuity of $\nabla f$ we have
\begin{align*}
\|z_{n-1} -\tilde{z}\| \leq \|x_n -\tilde{x}\| +\gamma\|\nabla f(x_n) -\nabla f(\tilde{x})\| \leq (1 +\gamma\kappa)\|x_n -\tilde{x}\| \leq (1 +\gamma\kappa)\Gamma_0\rho^\frac{n}{2}.    
\end{align*}
Therefore,
\begin{align*}
\|(x_n, y_n, z_n) -(\tilde{x}, \tilde{y}, \tilde{z})\| &\leq \|x_n -\tilde{x}\| +\|y_n -\tilde{y}\| +\|z_n -\tilde{z}\|\\
&\leq \left(3 +\gamma\kappa +\frac{2(1 +\gamma\kappa)}{\eta}\right)\Gamma_0\rho^\frac{n}{2}.    
\end{align*}
On the other hand, we derive from Lemma~\ref{l:Lfun}\ref{l:Lfun_lb}\&\ref{l:Lfun_ub} that
\begin{align*}
|L_n -F(y_n)| \leq \|y^*_n -y^*_{n-1}\| \|y_n -y_{n-1}\| +\frac{1}{2}\left(\frac{1}{\theta\gamma} +\kappa +\ell\right)\|y_n -x_n\|^2    
\end{align*}
By the boundedness of $(y^*_{n-1})_{n\in \mathbb{N}^\ast}$ (see Theorem~\ref{t:subcvg}\ref{t:subcvg_bounded}), \eqref{eq:zzyy}, and \eqref{eq:ynxn}, there exists $\Gamma_1\in (0, +\infty)$ such that, for all $n\in \mathbb{N}^\ast$, $|L_n -F(y_n)| \leq \Gamma_1\rho^\frac{n}{2}$. Noting that $\mathcal{L}(w_n) =L_n$ and $\mathcal{L}(\tilde{w}) =F(\tilde{y})$, we obtain that
\begin{align*}
|F(y_n) -F(\tilde{y})| \leq |\mathcal{L}(w_n) -\mathcal{L}(\tilde{w})| +|L_n -F(y_n)| \leq (\Gamma_0\rho^\frac{n}{2} +\Gamma_1)\rho^\frac{n}{2} \leq (\Gamma_0 +\Gamma_1)\rho^\frac{n}{2}.
\end{align*}
Letting $\Gamma=
\max\{\left(3 +\gamma\kappa +\frac{2(1 +\gamma\kappa)}{\eta}\right)\Gamma_0, \Gamma_0 +\Gamma_1\}$, we get \ref{t:fullcvg_linear}.

\emph{Case 3:} $\lambda >\frac{1}{2}$. Arguing as in the proof of \cite[Theorem~5.1(iv)]{BDL22} and \cite[Theorem~2(iii)]{Attouch2007}, there exists $\Gamma_0 \in (0, +\infty)$ such that, for all $n\in \mathbb{N}^\ast$,
\begin{align*}
\|x_n -\tilde{x}\| \leq \Gamma_0n^{-\frac{1-\lambda}{2\lambda-1}} \text{~~and~~} |\mathcal{L}(w_n) -\mathcal{L}(\tilde{w})|\leq \Gamma_0n^{-\frac{2-2\lambda}{2\lambda-1}}.  
\end{align*}
By the same steps as in \emph{Case 2} and noting that $(n+1)^{-\frac{1-\lambda}{2\lambda-1}} <n^{-\frac{1-\lambda}{2\lambda-1}}$ for $n\in \mathbb{N}^\ast$ as well as $(n-1)^{-\frac{1-\lambda}{2\lambda-1}} \leq 2^{\frac{1-\lambda}{2\lambda-1}}n^{-\frac{1-\lambda}{2\lambda-1}}$ for $n\geq 2$, we also have that, for all $n\in \mathbb{N}^\ast$,
\begin{align*}
\|(x_n, y_n, z_n) -(\tilde{x}, \tilde{y}, \tilde{z})\| \leq \left(3 +\gamma\kappa +\frac{2(1 +\gamma\kappa)}{\eta}\right)\Gamma_0 n^{-\frac{1-\lambda}{2\lambda-1}}
\end{align*}
and 
\begin{align*}
|F(y_n) -F(\tilde{y})| \leq (\Gamma_0 n^{-\frac{1-\lambda}{2\lambda-1}} +\Gamma_1)n^{-\frac{1-\lambda}{2\lambda-1}} \leq (\Gamma_0  +\Gamma_1)n^{-\frac{1-\lambda}{2\lambda-1}} \text{~~for some~} \Gamma_1\in (0, +\infty).
\end{align*}
Again letting $\Gamma=
\max\{\left(3 +\gamma\kappa +\frac{2(1 +\gamma\kappa)}{\eta}\right)\Gamma_0, \Gamma_0  +\Gamma_1\}$, we complete the proof.
\end{proof}

\begin{remark}[Comments on the assumption on $\ubar{h}$]
There are several previous studies in the literature \cite{An2016,Wen2017,Tan2023} employing the subgradient of the subtrahend part $\ubar{h}$. When it comes to proving the convergence of the whole sequence, they need an additional assumption that $\ubar{h}$ has Lipschitz continuous gradient. Inspired by the technique in \cite{TLiu2019}, we utilize the conjugate of $\ubar{h}$ to construct the Lyapunov function $\mathcal{L}$, and we do not require any additional assumptions on $\ubar{h}$ to prove the full sequential convergence, as shown in Theorem~\ref{t:fullcvg}.
\end{remark}

\section{Numerical results}
\label{sec:numerical_result}

All of the experiments are performed on a 64-bit laptop with Intel(R) Core(TM) i7-1165G7 CPU (2.80GHz) and 32GB of RAM with MATLAB R2023b.

\subsection{An analytical example}

Consider the toy example
\begin{align}\label{eq:toy}
\min_{x\in \mathbb{R}^d} \|Ax\|^2 +\rho\|x\|_1 +\frac{1}{2}e^{-\|x\|^2} -\rho\|x\|,
\end{align}
  where $A\in \mathbb{R}^{d\times d}$ and $\rho\in (0, +\infty)$. One stationary point of this problem is $y^* =0\in \mathbb{R}^d$. We will examine Algorithm~\ref{algo:DRFDR} in solving \eqref{eq:toy}. Let $f(x) =\|Ax\|^2$ ($\kappa =2\lambda_{\max}(A^\top A)$), $g(x) =\rho\|x\|_1$, $\bar{h}(x) =\frac{1}{2}e^{-\|x\|^2}$ ($\ell =e^{-2}$), and $\ubar{h}(x) =\rho\|x\|$. Here, $\lambda_{\max}(M)$ denotes the maximum eigenvalue of matrix $M$. The updating steps of the DRFDR in this case is
\begin{align*}
\begin{cases}
x_{n+1} &=(2\gamma A^\top A + I)^{-1}z_n,\\
y_{n+1} &=\prox_{\gamma\rho\|\cdot\|_1}(2x_{n+1} - z_n +2\gamma x_{n+1}e^{-\|x_{n+1}\|^2} + \gamma \rho y^*_n),\\
z_{n+1} &=z_n +\eta(y_{n+1} -x_{n+1}),
\end{cases}
\end{align*}
where $\prox_{\gamma\rho\|\cdot\|_1}$, known as the \emph{soft shrinkage} operator, is defined componentwise for all $u =(u_1, \dots, u_d)\in \mathbb{R}^d$ and all $i\in \{1, \dots, d\}$ as
\begin{align}\label{eq:proxL1}
(\prox_{\gamma\rho\|\cdot\|_1}(u))_i =\operatorname{sign}(u_i) \max\{0, \lvert u_i\rvert-\gamma\rho\}    
\end{align}
(see, e.g., \cite[Example~24.22]{BC17}), while
\begin{align}\label{eq:y*n}
y^*_n =\begin{cases}
0 & \text{if $y_n =0$},\\
\frac{y_n}{\|y_n\|} & \text{if $y_n \neq 0$}.
\end{cases}
\end{align}%

We run the DRFDR starting at the point $y_0=z_0=[10,10]^{\top}$, and the stopping condition is $\|y_{n+1}-y_n\|<10^{-3}$. Let us choose
\begin{align*}
A =\begin{bmatrix}
1 & 0 \\
0 & 0
\end{bmatrix}.
\end{align*}
For this case, $\alpha =0$. First, we fix $\gamma=0.22$ and let $\eta$ vary from $1.0$ to $1.5$. We then fix $\eta=1.5$ and let $\gamma$ vary from $0.05$ to $0.22$ (By Theorem~\ref{t:subcvg}, $\gamma <\overline{\gamma} =0.223$). From Figure~\ref{fig:eta_effect_gamma_eq_0}(a)\&(b), we observe that larger $\eta$ and $\gamma$ result in faster convergence.

\begin{figure}[H]%
\centering
\subfloat[]{{\includegraphics[width=0.45\linewidth]{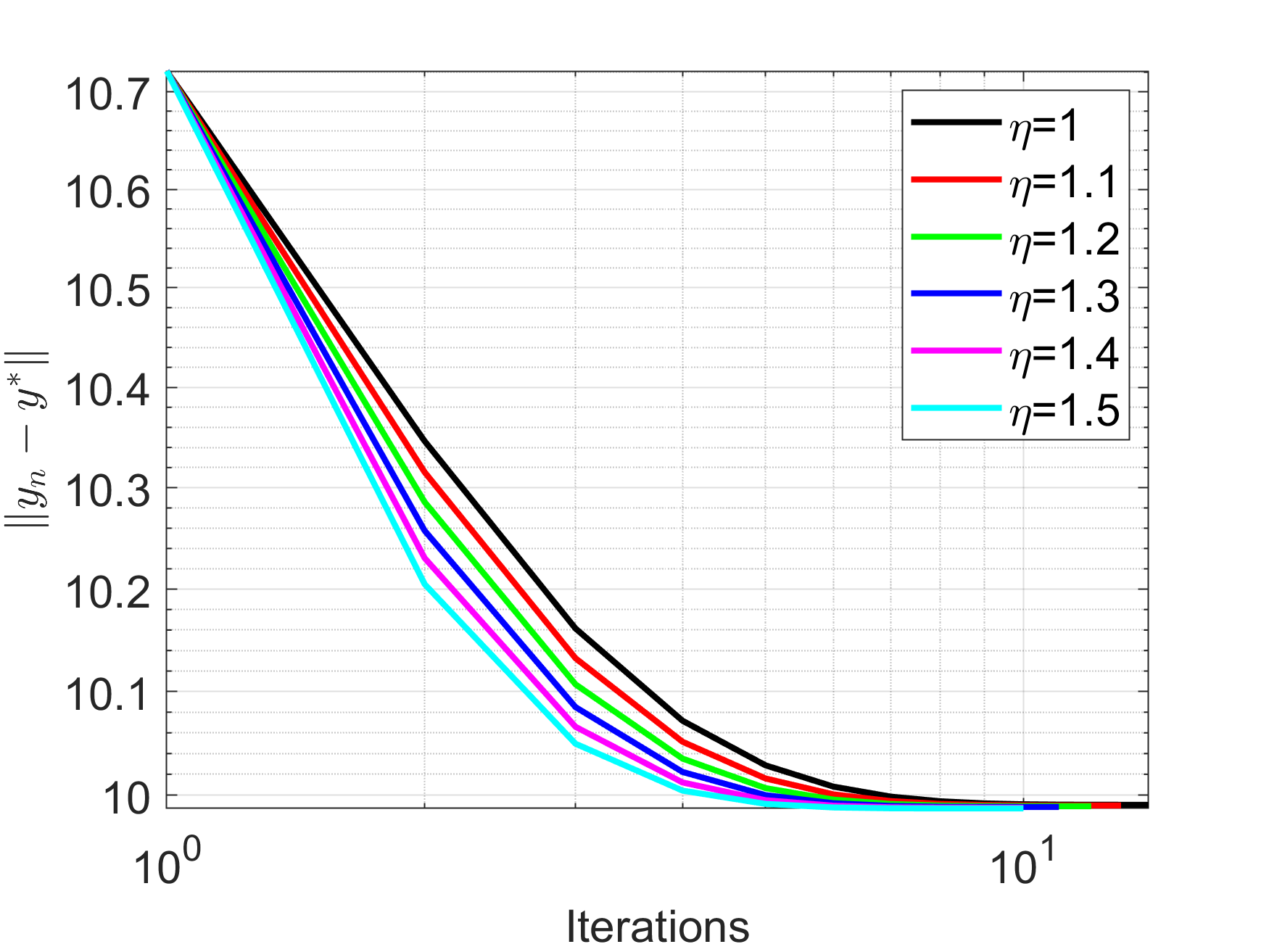} }}%
~
\subfloat[]{{\includegraphics[width=0.45\linewidth]{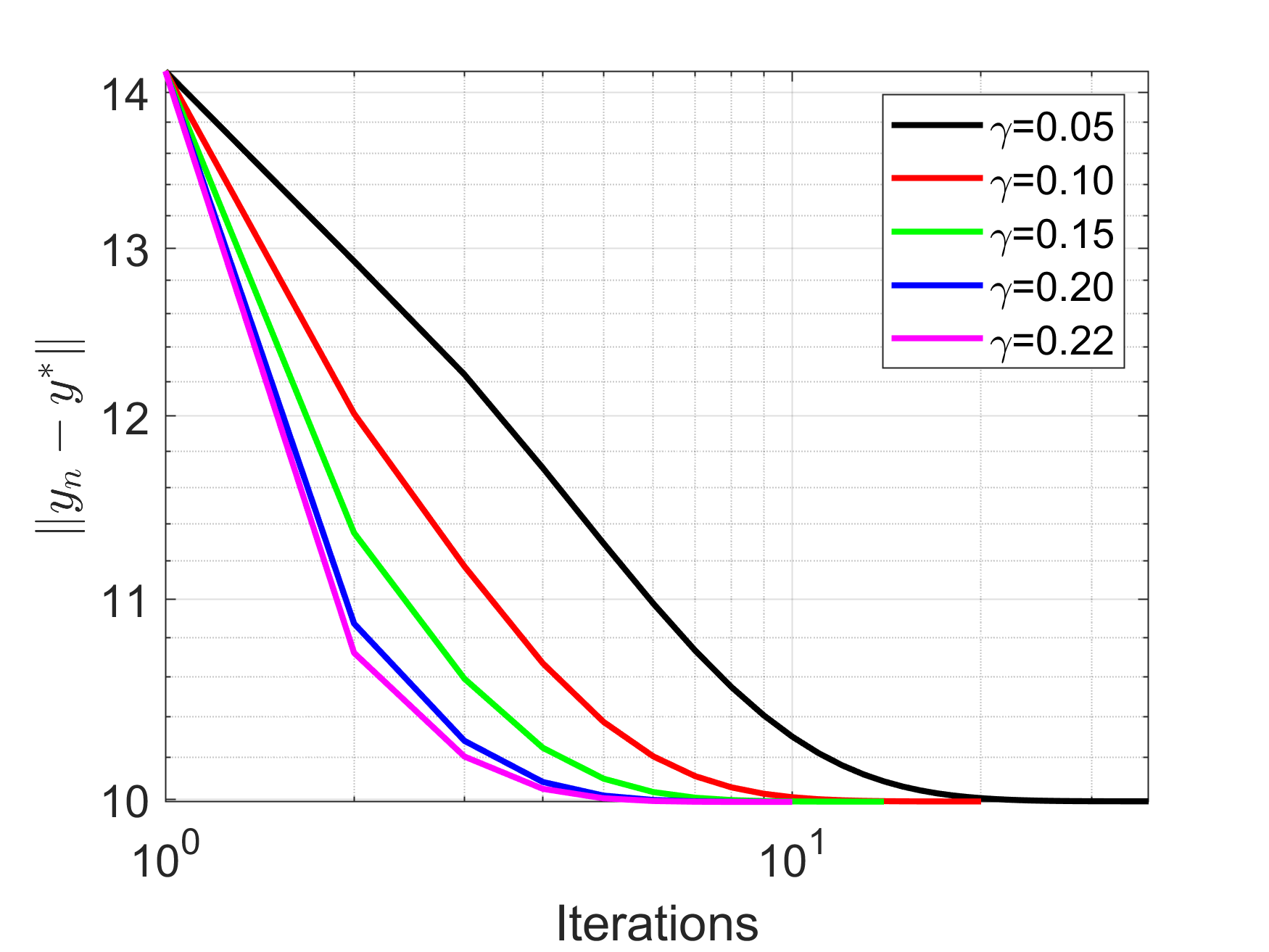} }}%
\caption{Effects of $\eta$ and $\gamma$.}%
\label{fig:eta_effect_gamma_eq_0}%
\end{figure}

Next, we choose $A \in \mathbb{R}^{2\times 2}$ to be the identity matrix. Then $\alpha =2 >0$ and we can choose $\eta \geq 2$. Let $\theta=1$, $\rho=0.1$, and $\gamma=0.3$. According to Theorem~\ref{t:subcvg}, $\eta \in [2, 3.87)$. Now, Figure~\ref{fig:eta_effect_2}(a) shows that larger $\eta$ results in faster convergence. We then fix $\eta=2$ and vary $\gamma$. Again it can be seen from Figure~\ref{fig:eta_effect_2}(b) that the DRFDR converges faster with larger $\gamma$.   

\begin{figure}[H]%
\centering
\subfloat[]{{\includegraphics[width=0.45\linewidth]{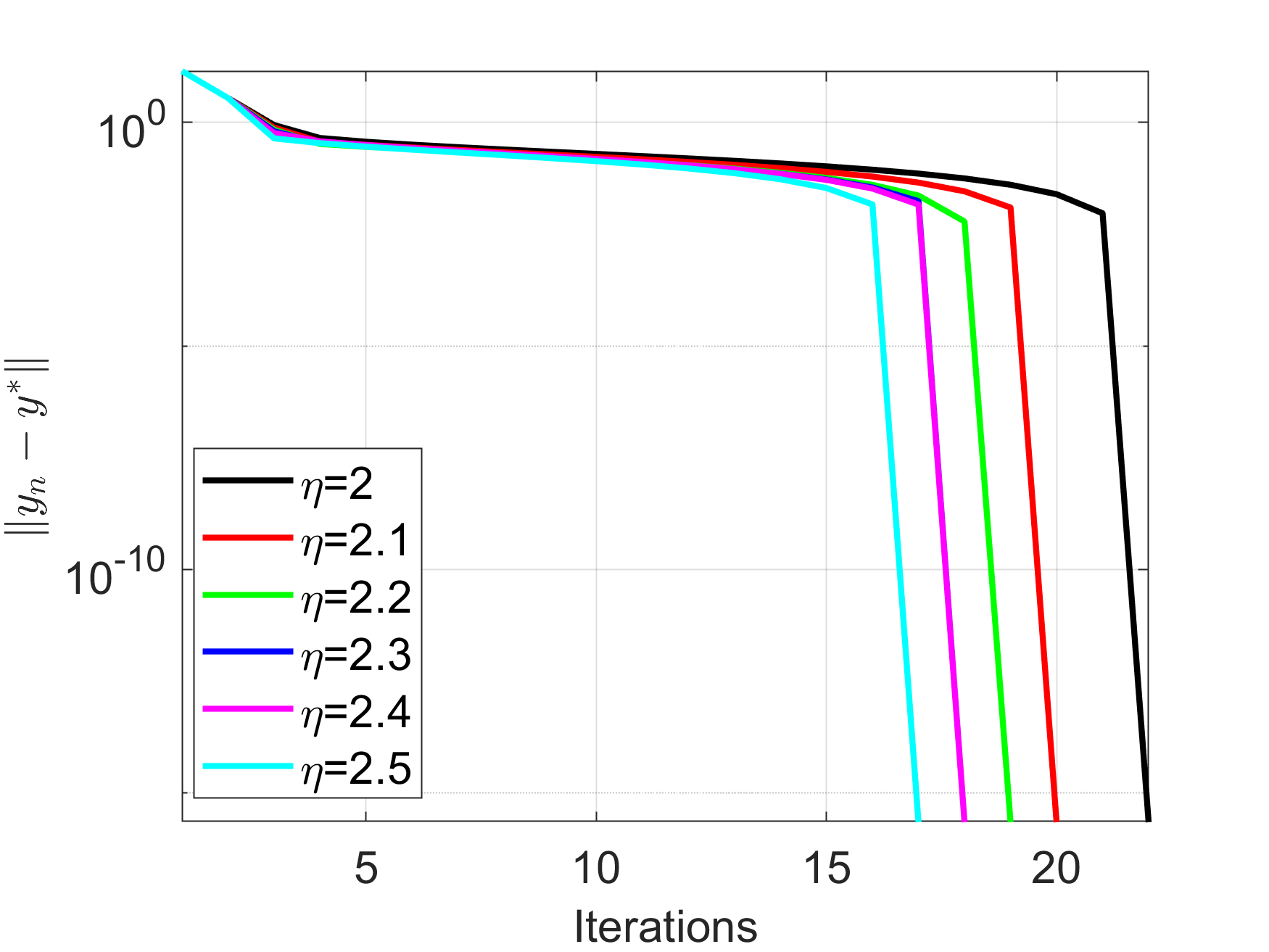} }}%
~
\subfloat[]{{\includegraphics[width=0.45\linewidth]{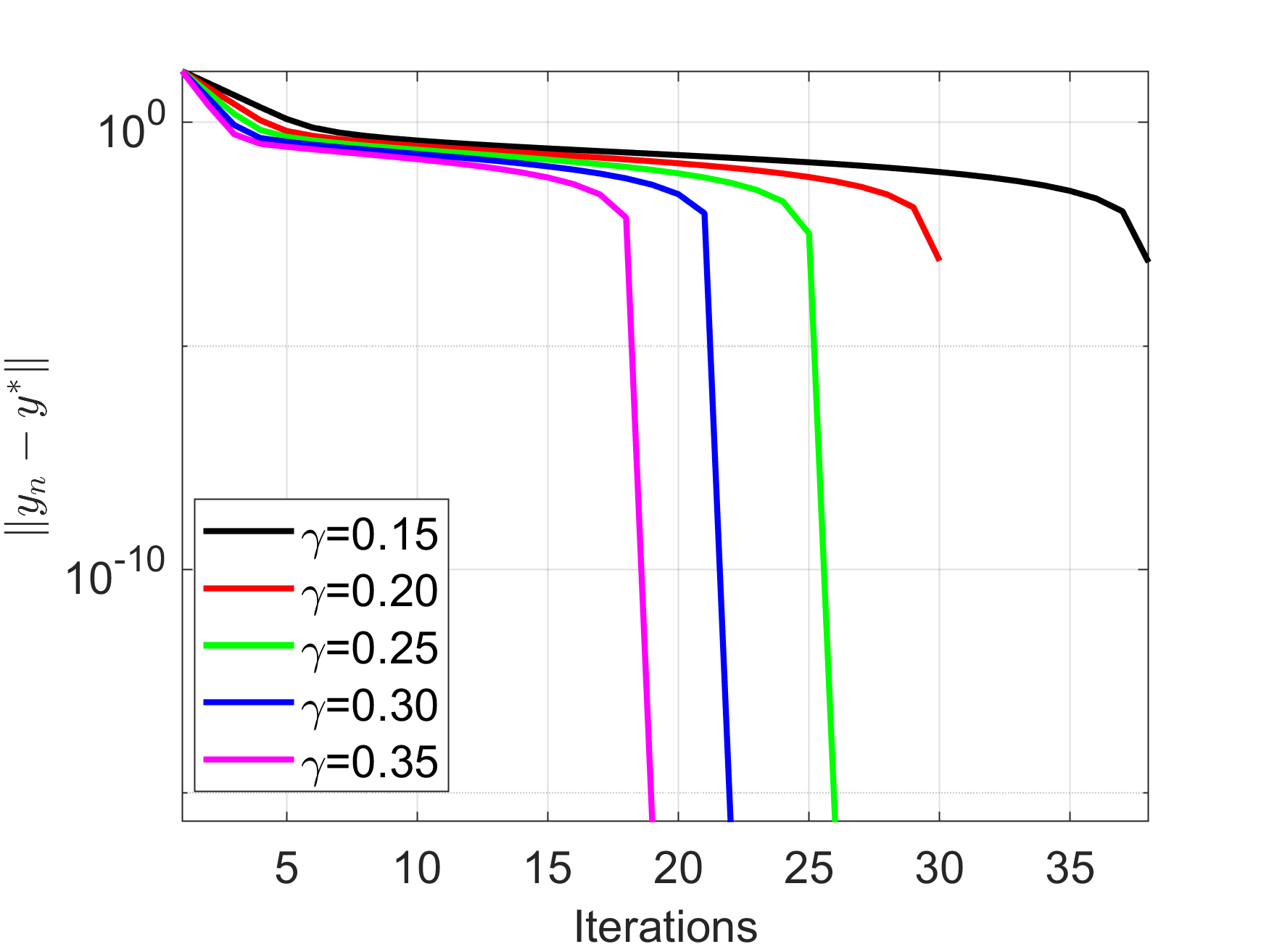} }}%
\caption{Effects of $\eta$ and $\gamma$ when $\alpha>0$.}%
\label{fig:eta_effect_2}%
\end{figure}

\subsection{Low-rank matrix recovery on synthetic data and real power system load data}
\label{subsec_LRMC}

We revisit the problem introduced in Example~\ref{ex:MAT_REC}. In this case, $f(X)=\frac{1}{2}\|\mathcal{P}_{\Omega} (X)- \mathcal{P}_{\Omega}(M)\|^2$, $g(X)=\iota_{\mathcal{C}(r)}(X)$, $\bar{h}(X)=\frac{\rho}{2}\|X\|_F^2$, and $\ubar{h} \equiv 0 $. We set $\theta=1$, $\alpha =0$, $\rho=1.8\times10^{-6}$ (since the value of the term $\bar{h}(X)$ is much larger compared to that of the term $f(X)$), $\ell=\rho$, and $\kappa=1$. The updating steps of Algorithm~\ref{algo:DRFDR} in this case become
\begin{align*}
\left\{
\begin{aligned}
&X_{n+1}^{ij} =\begin{cases}
\frac{1}{1+\gamma}(Z_n^{ij} +\gamma M^{ij}) &\text{if~} (i,j) \in \Omega,\\
Z_n^{ij} &\text{if~} (i,j) \notin \Omega,
\end{cases}\\
&Y_{n+1} =P_{\mathcal{C}(r)}((2-\gamma\rho)X_{n+1}-Z_n),\\
&Z_{n+1} =Z_n+\eta (Y_{n+1}-X_{n+1}),
\end{aligned}
\right.
\end{align*}
where $P_{\mathcal{C}(r)}(X)=U\sum_rV^\top=U\text{diag}(\{\sigma_i\}_{1\leq i \leq r})V^\top$. We employ the heuristic in \cite{LP16} to set the parameter $\gamma$. First, initialize $\gamma = k \gamma_0$ and update $\gamma=\max \{\gamma/2,0.9999\gamma\}$ whenever $\gamma>\gamma_0$ and the sequence $X_n$ satisfies $\|X_{n+1}-X_n\|_F>1000/n$ or $\|X_n\|_F>10^{10}$. We choose $\eta=1.8$, then in view of Theorem~\ref{t:subcvg}, $\overline{\gamma}\approx 0.32$, so we set $\gamma_0 =0.2$ for our algorithm. We compare our algorithm with the forward-backward splitting (FBS), the Douglas--Rachford splitting applied to \eqref{ex:mat_recover_re} (denoted as DRS) and to \eqref{ex:mat_recover_mod} (denoted as DRSR), and the Davis--Yin splitting (DYS), all of which are considered in \cite{Bian2021}. According to \cite{Bian2021}, $\gamma_0= 0.22$ for the DRS and DRSR, and $\gamma_0= 0.15$ for the DYS. The DRS, DRSR, and DYS used the same heuristic described above to update $\gamma$, while $\gamma=2/3$ for the FBS. We let $k=10^6$ for the DRS, DRSR, DYS, and the DRFDR. All the singular value decomposition (SVD) involved in this experiment were conducted by using PROPACK toolbox.

Let us now assess the performance of the DRFDR on randomly generated data. We generate $n \times  n$ matrices of rank $r$ by sampling two $n \times  r$ matrices $M_1$ and $M_2$ independently, each having i.i.d Gaussian entries, and let $M = M_1M_2^{\top}$. The set of observed entries $\Omega$  is sampled uniformly at random among all sets of cardinality $m$. The sampling ratio is defined as $\mathcal{R}=m/n^2$. We run all the algorithms, initialized at the sampled matrix, for a maximum of 2000 iterations. All the algorithms are run for 30 times, at each time $M$ and $\Omega$ are randomly generated. The stopping criterion for all algorithms is $\frac{\|\mathcal{P}_{\Omega}(Y_n-M)\|}{\|\mathcal{P}_{\Omega}(M)\|}<10^{-4}$ and the evaluation metric is relative error (RE)$:= \frac{\|Y_{\text{opt}}-M\|_F}{\|M\|_F}$. Table~\ref{tab:LRMC_random_data} reports the CPU time, the number of iterations, the RE at termination. The DRFDR outperforms all of the remaining algorithms in terms of CPU time and RE, and also requires fewer iterations to converge.

\begin{table}[!htbp]
\scriptsize
\caption{Results of 30 randomly generated instances, low-rank matrix recovery}
\label{tab:LRMC_random_data}
\begin{center}
\setlength{\tabcolsep}{1.2pt}
\resizebox{\columnwidth}{!}{
\begin{tabular}{cccccccccccccccc}
\hline
\multicolumn{1}{l|}{}      & \multicolumn{5}{c|}{CPU time (seconds)}                                                                                                    & \multicolumn{5}{c|}{Iterations}                                                                                                            & \multicolumn{5}{c}{RE}                                                                                                                    \\ \hline
\multicolumn{16}{c}{$\mathcal{R}=0.1$, rank 10}                                                                                                                                                                                                                                                                                                                                                                                                                            \\ \hline
\multicolumn{1}{c|}{Size$^*$}  & \multicolumn{1}{c}{FBS} & \multicolumn{1}{c}{DRSR} & \multicolumn{1}{c}{DRS} & \multicolumn{1}{c}{DYS} & \multicolumn{1}{c|}{\textbf{DRFDR}} & \multicolumn{1}{c}{FBS} & \multicolumn{1}{c}{DRSR} & \multicolumn{1}{c}{DRS} & \multicolumn{1}{c}{DYS} & \multicolumn{1}{c|}{\textbf{DRFDR}} & \multicolumn{1}{c}{FBS} & \multicolumn{1}{c}{DRSR} & \multicolumn{1}{c}{DRS} & \multicolumn{1}{c}{DYS} & \multicolumn{1}{c}{\textbf{DRFDR}} \\ \hline
\multicolumn{1}{c|}{5000}  & 122.3& 37.3& 162.8& 33.9& \multicolumn{1}{c|}{16.7}          & 178& 45& 212& 45& \multicolumn{1}{c|}{22}            & 1.20E-04& 1.02E-04& 1.22E-04& 9.87E-05& 8.81E-05
\\
\multicolumn{1}{c|}{8000}  & 297.8& 98.2& 396.6& 85.9& \multicolumn{1}{c|}{41.6}          & 163& 44& 194& 44& \multicolumn{1}{c|}{21}            & 1.13E-04& 1.03E-04& 1.12E-04& 1.01E-04& 8.20E-05
\\
\multicolumn{1}{c|}{12000} & 620.9& 214.9& 823.5& 193.3& \multicolumn{1}{c|}{92.9}          & 155& 44& 185& 44& \multicolumn{1}{c|}{21}            & 1.07E-04& 9.05E-05& 1.08E-04& 9.08E-05& 6.86E-05\\ \hline
\multicolumn{16}{c}{$\mathcal{R}=0.15$, rank 10}                                                                                                                                                                                                                                                                                                                                                                                                                           \\ \hline
\multicolumn{1}{c|}{5000}  & 74.0& 33.8& 100.5& 30.5& \multicolumn{1}{c|}{17.2}          & 107& 40& 129& 40& \multicolumn{1}{c|}{22}            & 1.11E-04& 8.53E-05& 1.08E-04& 9.22E-05& 7.65E-05
\\
\multicolumn{1}{c|}{8000}  & 182.9& 90.7& 255.5& 81.4& \multicolumn{1}{c|}{43.3}          & 100& 39& 122& 40& \multicolumn{1}{c|}{21}            & 1.09E-04& 9.62E-05& 1.03E-04& 8.36E-05& 8.37E-05
\\
\multicolumn{1}{c|}{12000} & 399.6& 201.2& 546.4& 176.8& \multicolumn{1}{c|}{95.5}         & 97& 39& 118& 39& \multicolumn{1}{c|}{21}            & 1.01E-04& 8.36E-05& 9.90E-05& 1.01E-04& 6.74E-05\\ \hline
\multicolumn{16}{c}{$\mathcal{R}=0.1$, rank 15}                                                                                                                                                                                                                                                                                                                                                                                                                            \\ \hline
\multicolumn{1}{c|}{5000}  & 155.2& 44.6& 207.2& 41.1& \multicolumn{1}{c|}{20.4}          & 193& 46& 230& 46& \multicolumn{1}{c|}{23}            & 1.26E-04& 1.03E-04& 1.27E-04& 1.00E-04& 9.95E-05
\\
\multicolumn{1}{c|}{8000}  & 367.5& 112.7& 473.2& 101.2& \multicolumn{1}{c|}{50.4}          & 173& 45& 204& 45& \multicolumn{1}{c|}{22}            & 1.16E-04& 9.93E-05& 1.17E-04& 9.60E-05& 7.57E-05
\\
\multicolumn{1}{c|}{12000} & 773.7& 252.0& 1029.0& 222.0& \multicolumn{1}{c|}{112.0}         & 161& 44& 192& 44& \multicolumn{1}{c|}{21}            & 1.13E-04& 1.01E-04& 1.09E-04& 1.00E-04& 8.04E-05\\ \hline
\multicolumn{16}{c}{$\mathcal{R}=0.15$, rank 15}                                                                                                                                                                                                                                                                                                                                                                                                                           \\ \hline
\multicolumn{1}{c|}{5000}  & 99.4& 41.0& 130.1& 38.1& \multicolumn{1}{c|}{21.2}          & 114& 40& 135& 40& \multicolumn{1}{c|}{22}            & 1.15E-04& 9.86E-05& 1.18E-04& 1.01E-04& 9.16E-05
\\
\multicolumn{1}{c|}{8000}  & 231.5& 103.2& 301.1& 94.7& \multicolumn{1}{c|}{52.5}          & 105& 40& 126& 40& \multicolumn{1}{c|}{22}            & 1.08E-04& 8.45E-05& 1.09E-04& 9.22E-05& 7.17E-05
\\
\multicolumn{1}{c|}{12000} & 479.0& 218.6& 627.7& 204.6& \multicolumn{1}{c|}{108.8}         & 100& 39& 121& 40& \multicolumn{1}{c|}{21}            & 1.04E-04& 9.66E-05& 1.02E-04& 8.38E-05& 7.96E-05\\ \hline
\multicolumn{16}{l}{$^*$Size $m$ means that the matrix $M$ is $m \times m$.}
\end{tabular}
}
\end{center}
\end{table}

Next, we evaluate the performance of the DRFDR on a real dataset. The dataset is taken from the ``2023 Distribution zone substation data'', which is publicly published by Ausgrid\footnote{The dataset can be found at \url{https://www.ausgrid.com.au/Industry/Our-Research/Data-to-share/Distribution-zone-substation-data}.}. We take the load data from five 11kV meters, namely ``Beacon Hill'', ``Botany'', ``City East'', ``Belrose'', and ``Surry Hill'', to form a $480 \times 364$ matrix. As studied in \cite{CSSE,Cai2018_IEEE,Wang2015}, load data and PMU data have low rank property, and can be compressed using SVD with negligible error. Hence, we compress our data into a 20-rank matrix using the SVD. We then evaluate the performance of all algorithms on this 20-rank matrix. Set $\gamma_0=0.22$ and $\eta=1.8$ for the DRFDR, while the parameters of the remaining algorithms are the same ones as before. Let $k=10$ for the DRS, DRSR, DYS, and the DRFDR. We run all the algorithms for 30 times, at each time we create a random $\Omega$. Table~\ref{tab:LRMC_load_data} shows the average results of 30 runs. It can be seen that on this real dataset, the DRFDR still outperforms all of the remaining algorithms, on different sampling ratios. Note that the sampling ratio in this case is $\mathcal{R}=m/(n_1n_2)$, where $(n_1, n_2) = (480, 364)$. In Figure~\ref{fig:load_plot}, we plot one representative row in the matrix recovered by the DRFDR (when $\mathcal{R}=0.4$) versus its corresponding ground truth value, and this figure shows that the accuracy of the recovering process is high.

\begin{table}[H]
\scriptsize
\caption{Results of 30 runs on power system load dataset}
\label{tab:LRMC_load_data}
\begin{center}
\setlength{\tabcolsep}{1.2pt}
\resizebox{\columnwidth}{!}{
\begin{tabular}{c|ccccc|ccccc|ccccc}
\hline
                         & \multicolumn{5}{c|}{CPU time (seconds)}                                                                                                     & \multicolumn{5}{c|}{Iterations}                                                                                                            & \multicolumn{5}{c}{RE}                                                                                                                    \\ \hline
\multicolumn{1}{c|}{$\mathcal{R}$} & \multicolumn{1}{c}{FBS} & \multicolumn{1}{c}{DRSR} & \multicolumn{1}{c}{DRS} & \multicolumn{1}{c}{DYS} & \multicolumn{1}{c|}{\textbf{DRFDR}} & \multicolumn{1}{c}{FBS} & \multicolumn{1}{c}{DRSR} & \multicolumn{1}{c}{DRS} & \multicolumn{1}{c}{DYS} & \multicolumn{1}{c|}{\textbf{DRFDR}} & \multicolumn{1}{c}{FBS} & \multicolumn{1}{c}{DRSR} & \multicolumn{1}{c}{DRS} & \multicolumn{1}{c}{DYS} & \multicolumn{1}{c}{\textbf{DRFDR}} \\ \hline
0.4& 21.47& 22.64& 20.76& 21.98& 13.55& 2000& 1863& 1816& 1981& 1220& 3.79E-02& 1.58E-02& 1.37E-02& 2.38E-02& 2.95E-03
\\
0.5& 21.06& 10.42& 9.07& 12.61& 4.99& 1839& 797& 749& 1064& 425& 3.07E-03& 2.20E-04& 1.99E-04& 2.43E-04& 1.61E-04
\\
0.6& 11.29& 6.08& 5.45& 6.26& 2.80& 910& 440& 427& 499& 220& 2.05E-04& 1.56E-04& 1.51E-04& 1.91E-04& 1.22E-04
\\
0.7& 4.37& 3.91& 3.61& 3.19& 1.12& 348& 269& 272& 244& 90& 1.74E-04& 1.29E-04& 1.28E-04& 1.45E-04& 1.12E-04\\ \hline
\end{tabular}
}
\end{center}
\end{table}

\begin{figure}[H]
\centering
\includegraphics[width=0.6\linewidth]{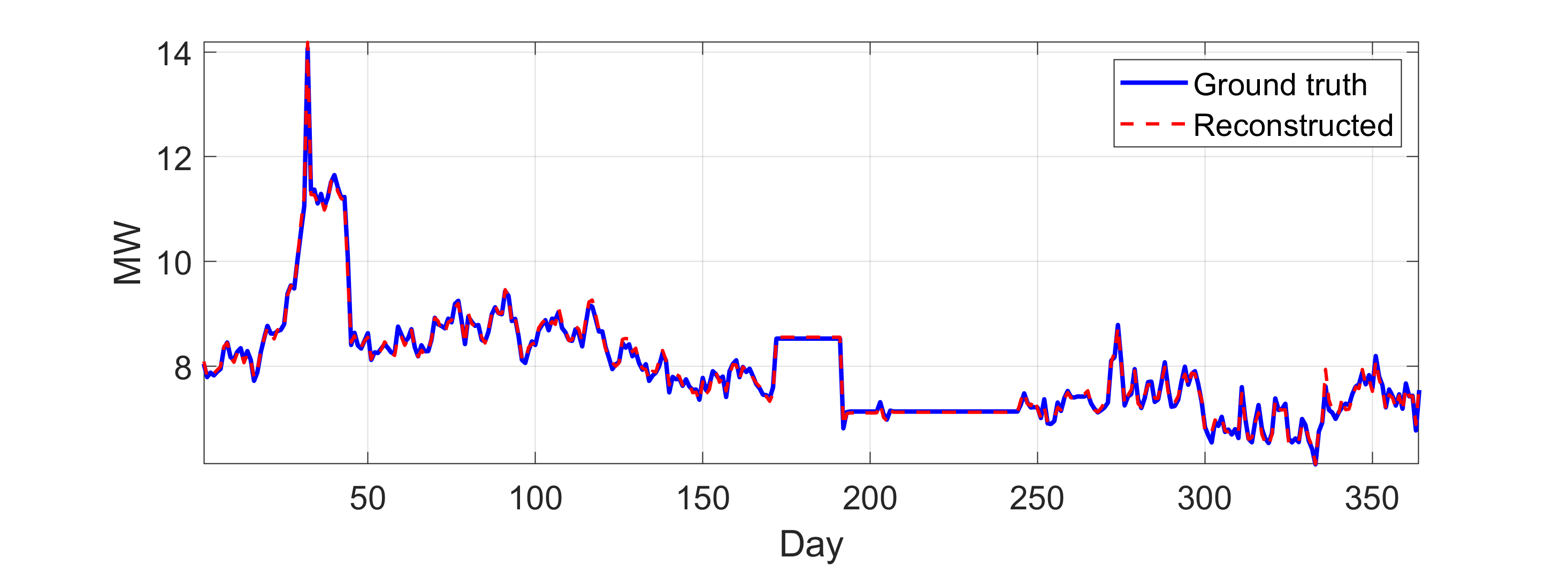}
\caption{Reconstructed load by DRFDR versus ground truth load.}
\label{fig:load_plot}
\end{figure}

\subsection{Image reconstruction via compressed sensing approach}

In this experiment, we aim to recover an image from an incomplete image, taking advantage of the sparsity in the image’s representation. Sparse representations are particularly advantageous in image reconstruction as they allow for efficient encoding of images using only a small number of significant coefficients. Such representations can be obtained via discrete cosine transform\footnote{\url{https://www.mathworks.com/help/images/discrete-cosine-transform.html}}. To reconstruct the image, one can formulate and solve the optimization problem
\begin{align}\label{ex:CSM}
\min_{x\in \mathbb{R}^{d}}~~\frac{1}{2}\|Ax-b\|^2 + \rho (\|x\|_1-\|x\|), \tag{CS}
\end{align}
where $\rho \in (0, +\infty)$ is a regularization parameter, $A\in \mathbb{R}^{m\times d}$ is a sensing matrix, and $b \in \mathbb{R}^{d}\smallsetminus \{0\}$. In this case, $A=S\Psi$ with $S$ being the sampling matrix and $\Psi$ is the inverse discrete cosine transform matrix, $b$ is the vectorized observed image. For a concrete example, given a ground truth image $I$ and its observed image with missing value $\Tilde{I}$ as
\begin{align*}
    I=\begin{bmatrix}
    i_1 & i_4 & i_7\\
    i_2 & i_5 & i_8 \\
    i_3 & i_6 & i_9
\end{bmatrix}
\text{~~and~~}
\Tilde{I}=\begin{bmatrix}
    i_1 & \times & i_7 \\
    \times & \times & \times \\
    \times & \times & i_9
\end{bmatrix}.
\end{align*}
Then $b=[i_1,i_7,i_9]^{\top}$ and 
\begin{align*}
S=\begin{bmatrix}
    1 & 0 & 0 & 0 & 0 & 0 & 0 & 0 & 0 \\
    0 & 0 & 0 & 0 & 0 & 0 & 1 & 0 & 0 \\
    0 & 0 & 0 & 0 & 0 & 0 & 0 & 0 & 1
\end{bmatrix}.    
\end{align*}
To use our proposed DRFDR algorithm on the problem \eqref{ex:CSM}, we let $f(x) = \frac{1}{2}\|Ax-b\|^2$, $g(x)=\rho \|x\|_1$, $\ubar{h}(x)=\rho\|x\|$, and $\bar{h} \equiv 0$ (hence $\ell=0$). From this, we set $\theta=1$, $\alpha =0$,  $\eta=1.8$, $\kappa=\lambda_{\max}(A^\top A)$, $\gamma$ is set similarly to the procedure in in Section~\ref{subsec_LRMC} with $k=200$. Now, the updating steps of the DRFDR read as
\begin{align*}
\begin{cases}
x_{n+1} &=(A^\top A + \frac{1}{\gamma}I)^{-1}(A^\top b +\frac{z_n}{\gamma}),\\
y_{n+1} &=\prox_{\gamma\rho\|\cdot\|_1}(2x_{n+1} - z_n + \gamma \rho y^*_n),\\
z_{n+1} &=z_n +\eta(y_{n+1} -x_{n+1}),
\end{cases}
\end{align*}
where $\prox_{\gamma\rho\|\cdot\|_1}$ and $y^*_n$ are as defined in \eqref{eq:proxL1} and \eqref{eq:y*n}.

We compare our algorithm with the unified Douglas--Rachford algorithm for DC programming (denoted as DRDC) in \cite{Chuang2021} and the ADMM with fast $L_1-L_2$ proximal operator in \cite{Lou2017}. Their settings are taken from their corresponding papers. Three test images are chosen for this experiment, namely ``MRI'', ``Cameraman'', and ``Moon'', which are all available inside MATLAB and have different pixel distributions. We set $\rho=10^{-4}$ and run all algorithms, initialized at zero, for a maximum of 3000 iterations. The stopping condition for all algorithms is $\frac{\|y_{n+1}-y_n\|}{\|y_n\|}<10^{-5}$. The experiment is repeated for 30 times. For each time, $\mathcal{R}$ percent of the entries in the image are randomly removed to create the observed image $\Tilde{I}$. After solving \eqref{ex:CSM}, the reconstructed image $\widehat{I}$ is obtained via reshaping $\Psi x$ into 2D form. This process is illustrated in Figure~\ref{fig:image_recon}. The evaluation metric used in this experiment is $RE:= \|I-\widehat{I}\|_F/\|I\|_F$. Table~\ref{tab:result_image} reports the results of this experiment, and it is seen that the DRFDR outperforms its competitors in terms of CPU time, and also slightly outperforms them in terms of RE.

\begin{figure}[htpb]
\centering
\includegraphics[width=\linewidth]{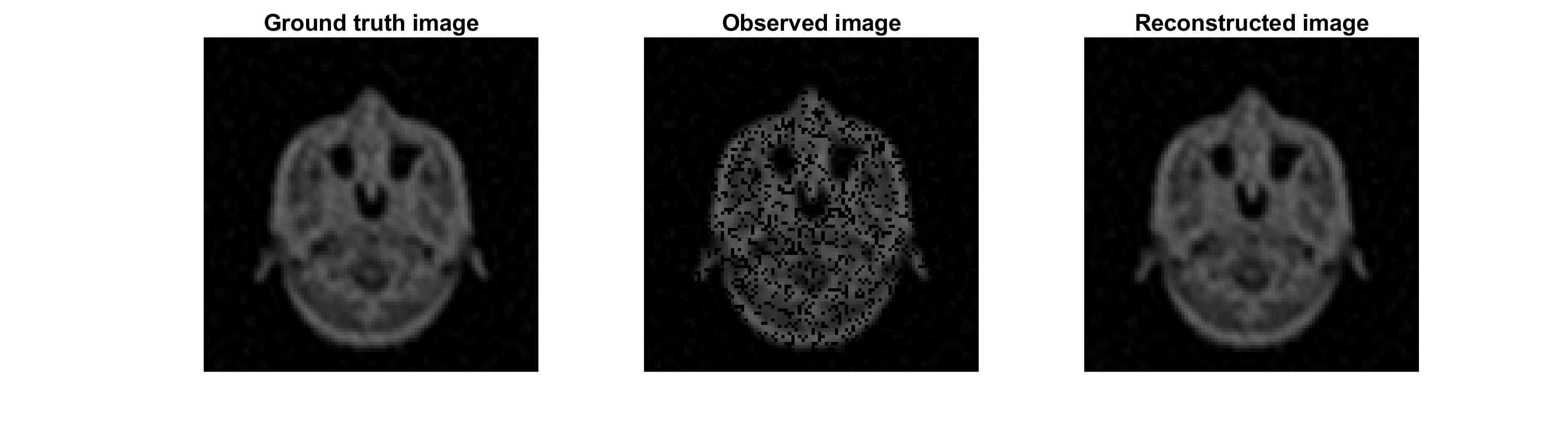}
\caption{An illustration of the experiment on ``MRI'' image, here $\mathcal{R}=25\%$.}
\label{fig:image_recon}
\end{figure}

\begin{table}[htpb]
\scriptsize
\caption{Results over 30 random runs, image reconstruction}
\label{tab:result_image}
\begin{center}
\setlength{\tabcolsep}{1.2pt}
\resizebox{\columnwidth}{!}{
\begin{tabular}{cccccccccc}
\hline
\multicolumn{1}{l|}{}              & \multicolumn{3}{c|}{CPU time (in seconds)}          & \multicolumn{3}{c|}{Iterations}                   & \multicolumn{3}{c}{RE}                 \\ \hline
\multicolumn{1}{c|}{$\mathcal{R}$} & ADMM  & DRDC  & \multicolumn{1}{c|}{\textbf{DRFDR}} & ADMM & DRDC & \multicolumn{1}{c|}{\textbf{DRFDR}} & ADMM      & DRDC      & \textbf{DRFDR} \\ \hline
\multicolumn{10}{c}{MRI}                                                                                                                                                              \\ \hline
\multicolumn{1}{c|}{10\%}          & 38.17 & 38.00 & \multicolumn{1}{c|}{27.13}          & 235  & 232  & \multicolumn{1}{c|}{170}            & 8.414E-03 & 8.403E-03 & 8.403E-03      \\
\multicolumn{1}{c|}{15\%}          & 38.06 & 36.01 & \multicolumn{1}{c|}{26.03}          & 255  & 242  & \multicolumn{1}{c|}{177}            & 1.086E-02 & 1.086E-02 & 1.086E-02      \\
\multicolumn{1}{c|}{20\%}          & 31.90 & 29.16 & \multicolumn{1}{c|}{22.00}          & 269  & 251  & \multicolumn{1}{c|}{184}            & 1.323E-02 & 1.324E-02 & 1.324E-02      \\
\multicolumn{1}{c|}{25\%}          & 27.25 & 24.94 & \multicolumn{1}{c|}{18.30}          & 283  & 260  & \multicolumn{1}{c|}{191}            & 1.572E-02 & 1.571E-02 & 1.571E-02      \\ \hline
\multicolumn{10}{c}{Cameraman}                                                                                                                                                        \\ \hline
\multicolumn{1}{c|}{10\%}          & 20.35 & 20.55 & \multicolumn{1}{c|}{16.88}          & 149  & 149  & \multicolumn{1}{c|}{125}            & 2.021E-03 & 2.020E-03& 2.020E-03\\
\multicolumn{1}{c|}{15\%}          & 22.06 & 20.93 & \multicolumn{1}{c|}{17.55}          & 167  & 160  & \multicolumn{1}{c|}{136}            & 2.519E-03 & 2.515E-03 & 2.513E-03      \\
\multicolumn{1}{c|}{20\%}          & 21.66 & 19.87 & \multicolumn{1}{c|}{17.27}          & 184  & 170  & \multicolumn{1}{c|}{149}            & 3.017E-03 & 3.014E-03 & 3.014E-03      \\
\multicolumn{1}{c|}{25\%}          & 21.43 & 19.50 & \multicolumn{1}{c|}{17.44}          & 200  & 182  & \multicolumn{1}{c|}{162}            & 3.690E-03 & 3.682E-03 & 3.681E-03      \\ \hline
\multicolumn{10}{c}{Moon}                                                                                                                                                             \\ \hline
\multicolumn{1}{c|}{10\%}          & 25.26 & 24.51 & \multicolumn{1}{c|}{19.58}          & 184  & 183  & \multicolumn{1}{c|}{144}            & 4.909E-03 & 4.909E-03 & 4.906E-03      \\
\multicolumn{1}{c|}{15\%}          & 23.29 & 22.39 & \multicolumn{1}{c|}{17.89}          & 200  & 193  & \multicolumn{1}{c|}{154}            & 6.297E-03 & 6.302E-03 & 6.297E-03      \\
\multicolumn{1}{c|}{20\%}          & 22.82 & 21.39 & \multicolumn{1}{c|}{17.20}          & 217  & 203  & \multicolumn{1}{c|}{164}            & 7.436E-03 & 7.433E-03 & 7.433E-03      \\
\multicolumn{1}{c|}{25\%}          & 21.74 & 20.04 & \multicolumn{1}{c|}{16.41}          & 229  & 211  & \multicolumn{1}{c|}{173}            & 8.770E-03 & 8.770E-03 & 8.770E-03      \\ \hline
\end{tabular}
}
\end{center}
\end{table}

\subsection{Simultaneously sparse and low-rank matrix estimation}

In this section, we consider the problem introduced in Example~\ref{ex:low_rank_and_sparse}. Let 
$f(X) =\frac{1}{2}\|X-A\|_F^2$ ($\kappa =1$), $g(X) =\rho_1 \|X\|_1$, $\bar{h}(X) =\rho_2 \|X\|_F^2$ ($\ell =2\rho_2$), and $\ubar{h}(X)=\rho_2\tnorm{X}^2_{k,2}$. We choose $\theta=1$, $\rho_1=\rho_2=0.1$, and $\eta=1.4$. The updating of the DRFDR in this case is
\begin{align*}
\left\{
\begin{aligned}
&X_{n+1} = \frac{\gamma}{1+\gamma}(A+Z_n), \\
&Y_{n+1} = \prox_{\gamma\rho_1\|\cdot\|_1}(2X_{n+1} -Z_n -\gamma \nabla \bar{h}(X_{n+1}) +\gamma Y^*_n), \\
&Z_{n+1} =Z_n +\eta(Y_{n+1} -X_{n+1}).
\end{aligned}
\right.
\end{align*}
where $\prox_{\gamma\rho_1\|\cdot\|_1}$ is the soft shrinkage operator imposed on all the entries of the input matrix. The subgradient $Y^*_n$ is calculated as described in \cite[Proposition~1]{Shi2022}. We consider two cases when $\alpha=0$ and $\alpha=1$. By Theorem~\ref{t:subcvg}, $\overline{\gamma}\approx 0.4167$ when $\alpha=0$, and $\overline{\gamma}\approx 0.7385$ when $\alpha=1$. Let $\gamma=\overline{\gamma} -10^{-20}$. We first generate a block diagonal matrix $A_G$, formed by five smaller, randomly generated square matrices of size $300\times 300$, $400\times 400$, $100\times 100$, $200\times 200$, and $100\times 100$. Each of them is formed by $vv^{\top}$ where the entries of $v$ are drawn i.i.d. from the uniform distribution on $[-1,1]$. Hence, all the submatrices have rank 1. We then randomly choose $\mathcal{R}$ percent of the entries of $A_G$ and add Gaussian noise to them to form the noisy input matrix $A$ (this means that $\mathcal{R}$ percent of the entries of $A_G$ are corrupted by noise). 

We compare the DRFDR with the following algorithms
\begin{itemize}
    \item Incremental Proximal Descent (IPD) in \cite{estimation_lrsm},
    \item Generalized Proximal Point Algorithm (GPPA) \cite{An2016}, which is a form of proximal DCA. To use the GPPA, we let $g_1(X)=\rho_1 \|X\|_1$, $g_2(X)= \frac{1}{2}\|X-A\|_F^2 + \rho_2 \|X\|_F^2$ (Lipschitz constant is $1+2\rho_2$), and $h(X)=\tnorm{X}^2_{k,2}$.
\end{itemize}
All algorithms are run for 30 runs over 30 randomly generated instances, initialized at $A$, and for each run we regenerate $A$. The maximum iteration is 2000 and the stopping condition for both algorithms is $\frac{\|Y_{n+1}-Y_n\|_F}{\|Y_n\|_F}<10^{-6}$. The proximal stepsize for the IPD is set to $0.5$. The evaluation metrics in this case is the relative error $\text{RE} := \frac{\|Y_{\text{opt}}-A_G\|_F}{\|A_G\|_F}$. Table~\ref{tab:Result_sparse_lowrank} shows that the DRFDR and the GPPA outperform the IPD, and it also shows that the new formulation \eqref{prob:SLRME_2}, which is solved by the DRFDR and GPPA, can yield competitive solutions compared to those of the conventional \eqref{prob:SLRME_1} solved by the IPD. It can also be seen that larger $\alpha$ yields larger $\gamma$, which leads to faster convergence and slightly better solutions for the DRFDR. The CPU time of the DRFDR is slightly longer than that of the GPPA, and this is due to the fact that the DRFDR has to perform more updating steps for its intermediate variables during the iteration process while the GPPA only performs one proximal step. It is also observed that the quality of the solutions obtained by the DRFDR are better than those obtained by the GPPA and IPD, especially when more entries of the matrix are corrupted by noise, i.e. higher $\mathcal{R}$.  

\begin{table}[H]
\scriptsize
\caption{Average results of 30 random generated instances, sparse \& low rank matrix estimation}
\label{tab:Result_sparse_lowrank}
\begin{center}
\setlength{\tabcolsep}{1.2pt}
\resizebox{\columnwidth}{!}{
\begin{tabular}{c|lccc|lccc|lccc}
\hline
\multicolumn{1}{l|}{} &  &\multicolumn{3}{c|}{CPU time (seconds)} &  &\multicolumn{3}{c|}{Iterations} &  &\multicolumn{3}{c}{RE} \\ \hline
\multirow{2}{*}{$\mathcal{R}$}                     &  \multirow{2}{*}{GPPA}&\multirow{2}{*}{IPD}           & \textbf{DRFDR}  & \textbf{DRFDR}        &  \multirow{2}{*}{GPPA}&\multirow{2}{*}{IPD}            & \textbf{DRFDR}  & \textbf{DRFDR}            &  \multirow{2}{*}{GPPA}&\multirow{2}{*}{IPD}        & \textbf{DRFDR}  & \textbf{DRFDR}    \\
 ~&  &~& ($\alpha=0$) & ($\alpha=1$) &  &~ &  ($\alpha=0$) & ($\alpha=1$)~ &  &~&  ($\alpha=0$) & ($\alpha=1$) 
\\ \hline
15\%                  &  2.01&4.18& 3.04&   2.39&  8&17& 12&  10&  6.05E-01&7.17E-01& 6.66E-01& 5.71E-01
\\
20\%                  &  2.07&4.56& 3.15& 2.68&  8&17& 12&   10&  6.74E-01&8.00E-01& 6.83E-01& 6.19E-01
\\
25\%                  &  2.03&4.56& 3.08&  2.59&  8&17& 12&   10&  7.76E-01&9.22E-01& 7.10E-01&  6.95E-01
\\
30\%                  &  1.83&4.08& 2.81&2.33&  8&17& 12&10&  8.44E-01&1.00E+00& 7.49E-01& 7.31E-01
\\

 35\%                  & 1.95& 3.93& 2.80& 2.28& 7& 17& 12& 10& 9.02E-01& 1.07E+00& 7.47E-01&7.29E-01\\
 \hline
\end{tabular}
}
\end{center}
\end{table}

\section{Conclusion}
\label{sec:conclusion}

We have proposed a splitting algorithm for minimizing the sum of three nonconvex functions, one of which is expressed in a DC form. The sequence of iterates generated by the proposed algorithm is bounded and any of its cluster points is a stationary point of the model problem. We have also established the global convergence of the whole sequence and derived the convergence rates of both the iterates and objective function values by further assuming the KL property for a suitable merit function, without any additional assumption on the concave part in the objective function. Our analysis not only extends the scope but also unifies and enhances recent convergence analyses of the Douglas--Rachford, Peaceman--Rachford, and David--Yin splitting algorithms in nonconvex settings. Intensive numerical experiments on both synthetic data and real data shows the superiority of our proposed algorithm. The experiment on the real Ausgrid dataset suggests that splitting algorithms have strong potential in developing data recovery methods for power system operators, providing a proof-of-concept for future research. The experiment on the simultaneously sparse and low-rank matrix estimation with the new nonconvex formulation also shows promising results.

\paragraph{Acknowledgements.}

The authors are grateful to Guoyin Li for his warm comments on an early version of this work and to Fengmiao Bian for sharing the MATLAB code used in \cite{Bian2021}. We are grateful to the anonymous reviewers for their constructive comments and suggestions which helped improve the manuscript.

\section*{Declarations}

\noindent\textbf{Funding.} MND was partially supported by Discovery Project DP230101749 from the Australian Research Council (ARC) and by a public grant from the Fondation Math\'ematique Jacques Hadamard (FMJH). TNP was supported by Henry Sutton PhD Scholarship Program from Federation University Australia. This research is also funded by Thuongmai University, Hanoi, Vietnam.

\noindent\textbf{Conflict of interest.} The authors declare no competing interests.

\noindent\textbf{Data availability.} All data generated or analyzed during this study are included in this article.

\end{document}